\documentclass[11pt,reqno]{amsart}

\usepackage[margin=1.15in]{geometry}
\usepackage[T1]{fontenc}
\usepackage[utf8]{inputenc}
\usepackage{lmodern}
\usepackage{microtype}
\usepackage{mathtools,amssymb,amsfonts,amsmath,amsthm}
\usepackage{bm,mathrsfs}
\usepackage{booktabs}
\usepackage{array}
\newcolumntype{L}[1]{>{\raggedright\arraybackslash}p{#1}}
\usepackage{enumitem}
\usepackage{xcolor}
\usepackage{graphicx}
\usepackage{url}
\usepackage[numbers,sort&compress]{natbib}
\usepackage[colorlinks=true,linkcolor=blue!55!black,citecolor=blue!55!black,urlcolor=blue!55!black]{hyperref}
\usepackage[nameinlink,capitalize,noabbrev]{cleveref}

\allowdisplaybreaks
\numberwithin{equation}{section}
\setlist[itemize]{leftmargin=1.5em,itemsep=0.25em,topsep=0.35em}
\setlist[enumerate]{leftmargin=1.8em,itemsep=0.25em,topsep=0.35em}

\newtheorem{theorem}{Theorem}[section]
\newtheorem{proposition}[theorem]{Proposition}
\newtheorem{lemma}[theorem]{Lemma}
\newtheorem{corollary}[theorem]{Corollary}
\newtheorem{assumption}[theorem]{Assumption}

\crefname{assumption}{assumption}{assumptions}
\Crefname{assumption}{Assumption}{Assumptions}
\theoremstyle{definition}

\newtheorem{example}[theorem]{Example}
\theoremstyle{remark}
\newtheorem{remark}[theorem]{Remark}

\newcommand{\X}{\mathsf X}
\newcommand{\Pcal}{\mathcal P}

\newcommand{\A}{\mathsf A}
\newcommand{\Q}{\mathsf Q}
\newcommand{\G}{\mathsf G}
\newcommand{\K}{\mathsf K}
\newcommand{\D}{\mathsf D}
\newcommand{\PhiPot}{\Phi}
\newcommand{\divpi}{\operatorname{div}_{\pi}}
\newcommand{\KL}{\operatorname{KL}}
\newcommand{\Law}{\operatorname{Law}}
\newcommand{\supp}{\operatorname{supp}}
\newcommand{\Tr}{\operatorname{Tr}}
\newcommand{\osc}{\operatorname{osc}}
\newcommand{\Lip}{\operatorname{Lip}}
\newcommand{\Dom}{\operatorname{Dom}}
\newcommand{\Id}{\mathrm{Id}}
\newcommand{\R}{\mathbb R}
\newcommand{\N}{\mathbb N}
\newcommand{\E}{\mathbb E}
\newcommand{\dd}{\,\mathrm d}
\newcommand{\ip}[2]{\left\langle #1,#2\right\rangle}
\newcommand{\norm}[1]{\left\lVert #1\right\rVert}
\newcommand{\abs}[1]{\left\lvert #1\right\rvert}
\newcommand{\eps}{\varepsilon}

\newcommand{\Wone}{W_1}

\newif\ifdraftnotes
\draftnotesfalse

\title[Uniform-in-time chaos for Green--Bessel SVGD]{Target-Adapted Green--Bessel SVGD:\\ Uniform-in-Time Propagation of Chaos and Last-Iterate Consistency}

\author{Trevor Teolis}
\author{Maarten V. de Hoop}
\thanks{Both authors: Department of Computational Applied Mathematics and Operations Research, Rice University, Houston, Texas, USA. Email: \href{mailto:tt111@rice.edu}{\texttt{tt111@rice.edu}} (Trevor Teolis); \href{mailto:mvd2@rice.edu}{\texttt{mvd2@rice.edu}} (Maarten V. de Hoop).}

\makeatletter
\@ifundefined{subjclassname@2020}{\@namedef{subjclassname@2020}{\textup{2020} Mathematics Subject Classification}}{}
\makeatother
\subjclass[2020]{Primary 60K35, 35Q70, 65C35; Secondary 46E35, 47D07, 68T07}
\keywords{Stein variational gradient descent, propagation of chaos, target-adapted kernels, Langevin generator, negative Sobolev discrepancy, interacting particles, last-iterate consistency}

\begin{document}

\begin{abstract}
We prove uniform-in-time propagation of chaos and last-iterate consistency for a target-adapted Stein variational gradient descent (SVGD) flow on compact connected manifolds.  The target has a smooth positive density, and the particles start independently from a fixed smooth nonnegative density ratio.  The construction uses the Green--Bessel operator $\Q_{r,\pi}=\A_\pi^{-1}(\Id+\A_\pi)^{-r}$ of the reversible target Langevin generator.  Sufficient Bessel smoothing gives a scalar kernel with finite diagonal, and a positive matrix lift realizes its potential force as a Stein velocity.  Population and empirical flows then dissipate the same finite target discrepancy.  Population entropy and the target spectral gap give decay of this discrepancy; a finite-time particle comparison reaches a time after which common-energy monotonicity controls every later time.  The resulting expected uniform discrepancy is $O((\log N)^{-1/2})$, with a corresponding logarithmic $W_1$ bound and consistency along every sequence $t_N\to\infty$.  We also prove an exact finite-mode approximation theorem with an explicit spatial-resolution error and a feature-factorized particle implementation.  For confining Euclidean targets, we establish static kernel and moment results and give a conditional dynamical extension under explicit population-regularity and transport hypotheses.
\end{abstract}

\maketitle

\section{Introduction}
\label{sec:introduction}

Can a deterministic particle sampler remain accurate at arbitrarily late times as the number of particles grows?  For Stein variational gradient descent (SVGD), the population flow and its particle approximation pose different parts of this question.  The population may converge to the target, while the error in a finite-time particle comparison grows with the observation horizon.  We construct a target-adapted Stein kernel for which a common energy joins these two statements into an all-time guarantee on compact manifolds.

SVGD moves interacting particles $X_1^N(t),\ldots,X_N^N(t)$ using a velocity obtained from a reproducing kernel and the target Stein operator \citep{liu2016stein,liu2017stein}.  Their empirical measure is
\[
\mu_t^N=\frac1N\sum_{i=1}^N\delta_{X_i^N(t)}.
\]
Replacing this measure by a deterministic law $\rho_t$ gives the population continuity equation.  For regular kernels, mean-field theory compares $\mu_t^N$ and $\rho_t$ on fixed intervals \citep{lu2019scaling,korba2020nonasymptotic,duncan2023geometry}.  A bound of order $C_T/\sqrt N$ can also justify growing horizons $t_N$ when $C_{t_N}/\sqrt N\to0$; restricted growing-time results are available for SVGD \citep{carrillo2025convergence}.  The issue is to control times beyond that restriction.

Our compact theorem proves
\begin{equation}
\label{eq:intro-physical-time}
\E\sup_{t\ge0}\Wone(\mu_t^N,\rho_t)\longrightarrow0.
\end{equation}
Together with population convergence, this gives $\E W_1(\mu_{t_N}^N,\pi)\to0$ for every deterministic sequence $t_N\to\infty$.  It concerns the particle cloud at the selected time.  The supremum inside the expectation in \eqref{eq:intro-physical-time} controls the whole random trajectory; the weaker ordering $\sup_t\E W_1\to0$ would already suffice for deterministic observation times.

\subsection{One energy for particles and population}

Population entropy alone cannot supply a configuration-wise particle argument: for a smooth positive target $\pi$, the empirical measure is atomic and $\KL(\mu_t^N\mid\pi)=+\infty$.  Entropy of the joint particle law remains useful for other particle guarantees \citep{banerjee2026improved}.  Here we use a target discrepancy that is finite for each particle configuration and decreases along both flows.  Entropy is then needed only to make the population approach the target in that discrepancy.

The construction starts from the nonnegative target Langevin generator $\A_\pi=-\operatorname{div}_\pi\nabla$ in $L^2(\pi)$, where $\operatorname{div}_\pi F=\pi^{-1}\nabla\cdot(\pi F)$.  On the subspace $L^2_0(\pi)$ of centered functions, set
\begin{equation}
\label{eq:intro-Q}
\Q_{r,\pi}=\A_\pi^{-1}(\Id+\A_\pi)^{-r}.
\end{equation}
The Green factor $\A_\pi^{-1}$ adapts the discrepancy to the target's spectral gap.  The Bessel factor smooths its scalar kernel $\G_{r,\pi}$ enough to give finite diagonal energy and a regular force on point masses.  We use the corresponding maximum mean discrepancy (MMD) \citep{sriperumbudur2010hilbert}:
\begin{equation}
\label{eq:intro-discrepancy}
\D_{r,\pi}(\mu,\nu)^2
=\iint\G_{r,\pi}(x,y)\,(\mu-\nu)(\dd x)(\mu-\nu)(\dd y).
\end{equation}
The precise measure domain and its dual formulation are given in \Cref{sec:framework}.

The Wasserstein gradient velocity of one half this squared MMD is the negative gradient of the error potential, as in the MMD-flow framework of \citet{arbel2019mmd}.  To realize this velocity as SVGD, we construct a positive matrix kernel in the framework of \citet{wang2019matrix} whose source Stein divergence is $-\nabla_x\G_{r,\pi}$.  Thus the target energy determines the force, and the matrix lift makes that force a Stein direction.  \Cref{sec:kernel-construction} proves the identity and \Cref{sec:energy-population} proves dissipation for both smooth laws and atomic measures.

\subsection{Why a finite-time comparison is enough}

Three facts drive the proof.  Both flows dissipate $\D_{r,\pi}(\cdot,\pi)^2$.  Population entropy dissipation and the target Poincar\'e gap give $\D_{r,\pi}(\rho_T,\pi)\le CT^{-1/2}$.  Finally, smooth compact-state interactions give a particle--population comparison of order $Ce^{CT}/\sqrt N$ on $[0,T]$.

The common energy controls what happens after that comparison interval.  For every $t\ge T$, the triangle inequality and target-energy monotonicity give
\begin{equation}
\label{eq:intro-stitching}
\D_{r,\pi}(\mu_t^N,\rho_t)
\le\D_{r,\pi}(\mu_T^N,\rho_T)
+2\D_{r,\pi}(\rho_T,\pi).
\end{equation}
Choose the deterministic comparison time $T$ so that the population term is small, then take $N$ large at that fixed $T$.  No later particle comparison is needed.  Choosing $T$ proportional to $\log N$ also gives the quantitative rate in the main theorem.  The observation time $t$ itself remains unrestricted.

\subsection{Results and scope}

\Cref{thm:compact-main} gives the complete result for a smooth positive target on a compact connected manifold, sufficiently large integer $r$, and i.i.d. initialization from a smooth nonnegative density ratio.  This initial law is fixed as $N$ grows and may be far from the target.  The uniform discrepancy bound is $C[\log(2+N)]^{-1/2}$; an interpolation estimate gives the $W_1$ bound $C[\log(2+N)]^{-1/[2(r+2)]}$.  Every fixed labelled marginal converges to the corresponding product target along every sequence $t_N\to\infty$.

The smoothing and compactness hypotheses have distinct purposes.  Smoothing makes the energy and force meaningful on empirical measures.  Compactness supplies global force bounds and a quantitative passage from discrepancy to $W_1$; connectedness supplies the positive target gap used for population attraction.  The proof establishes the required flow and comparison estimates from these conditions.

\Cref{thm:finite-mode-main} retains finitely many exact target eigenmodes.  The truncated dynamics preserves the same identities, with an additional spatial-resolution error of order $\Lambda^{-1/2}$ at spectral cutoff $\Lambda$.  A growing cutoff gives a joint particle--mode--time consistency theorem.  Once $L$ mode values and their gradients have been evaluated, the feature factorization costs $O(NLd)$ per velocity evaluation.  The theorem concerns exact modes; approximate eigenpairs require control of the generator residual described in \Cref{sec:approx-eigenpairs}.

On $\R^d$, the confinement conditions in \Cref{ass:euclidean-target} give local kernels, moment control, separation of measures, and global existence of finite empirical systems.  The all-time dynamical extension in \Cref{thm:euclidean-main} additionally assumes population regularity and a precise transport bound.  Common-covariance Gaussian mixtures satisfy the static conclusions and an explicit gap estimate; their dynamical conclusion has the same additional requirements.  This extension identifies which parts of the compact argument survive on an unbounded state space.

\subsection{Relation to existing work}

The inverse-generator construction is motivated by Laplacian adjusted Wasserstein gradient descent (LAWGD) \citep{chewi2020svgd}, which establishes strong population convergence using the target Langevin spectrum.  Our construction combines its target adaptation with Bessel smoothing, an exact positive matrix Stein lift, and a discrepancy finite on atoms.  Spectral implementation also has direct predecessors in diffusion-map particle systems \citep{li2023diffusion} and Koopman spectral methods \citep{xu2026koopman}; these methods motivate the spectral approximation questions addressed in \Cref{sec:finite-truncation}.  The operator-geometric perspective of \citet{sakthivadivel2026two} supplies complementary context for prescribed response operators.

The MMD interpretation and its energy-dissipation identity belong to the framework of \citet{arbel2019mmd}; matrix-valued Stein directions were developed by \citet{wang2019matrix}.  KSD descent \citep{korba2021ksd} is another Wasserstein descent of a finite empirical discrepancy.  The two scalar discrepancies in our construction have different roles: the squared Green--Bessel MMD has spectral weights $q_r(\lambda)=[\lambda(1+\lambda)^r]^{-1}$, whereas the squared KSD of the matrix kernel has weights $\lambda q_r(\lambda)=(1+\lambda)^{-r}$ and equals the population entropy dissipation.  The spectral-gap comparison is proved in \Cref{sec:energy-population}.

Regularized MMD flows provide close convergence comparisons.  The DrMMD analysis of \citet{chen2025drmmd} includes both population convergence and a finite-particle terminal-time $W_2$ estimate (Theorem~6.1), with particle and target-sample requirements depending on the chosen horizon and with adaptive regularization.  Sobolev-regularized MMD \citep{tian2026sobolev} uses current-measure-dependent regularization and gives population MMD decay up to a source-dependent residual; its population time-discretization theorem does not establish convergence of the particle implementation.  For Sobolev kernels on the torus, \citet{chizat2026kernel} prove quantitative population convergence, with local smallness and Sobolev assumptions in the smooth regime $s>1$.  Their multiplier $\lambda^{-s}$ has the same high-frequency order as ours when $s=r+1$.  For energy kernels, \citet{rosenzweig2026energy} prove finite-horizon modulated-energy mean-field estimates and, in Corollary~4.9, all-time consistency for initial particles already approaching the target.  Our compact theorem permits a fixed smooth initial law away from the target.

The finite-time mean-field theory \citep{lu2019scaling,korba2020nonasymptotic}, geometric analysis \citep{duncan2023geometry}, and variational many-particle and long-time limits \citep{nusken2023many} describe complementary aspects of SVGD.  Quantitative finite-particle results include averaged or best-iterate KSD bounds \citep{shi2023finite} and the joint-law entropy approach of \citet{banerjee2026improved}.  Regularized SVGD interpolates between Stein and Wasserstein geometries \citep{he2025regularized}; its finite-particle analysis \citep{he2026finite} gives Fisher and Wasserstein conclusions for specified annealed or time-averaged objects, which differ from the random last iterates studied here.  The companion Riesz particle analysis \citep{teolis2026riesz} uses renormalized entropy to control time-averaged empirical laws and invariant laws.

Population convergence itself has several distinct mechanisms.  The Stein--log-Sobolev inequality of \citet{carrillo2024stein} gives exponential convergence for target-weighted kernels with a translation-invariant convolution factor.  For periodic Riesz kernels, \citet{chizat2026stein} prove local quantitative population convergence in strong norms under entropy smallness and Sobolev control.  Our population estimate uses an ordinary Poincar\'e gap and a weak discrepancy finite on empirical measures.

Full empirical-Wasserstein control in physical time already exists for Gaussian target and Gaussian initialization with bilinear kernels: \citet[Theorem~3.7]{liu2023gaussian} place the supremum outside expectation.  \citet[Corollary~5.9]{balasubramanian2026uniform} obtain $\E\sup_{t\ge0}W_2^2(\mu_t^N,\rho_t)\le Cr_d(N)$ in that setting.  Their broader distributional results concern time averages, while their non-Gaussian initialization and conjugacy theory controls closed Stein features.  Our theorem treats smooth targets on compact manifolds using an infinite-rank kernel and a fixed smooth initial density.

Added Langevin noise gives another route.  \citet{priser2025noisy} prove non-averaged long-time consistency for noisy discrete-time SVGD under a log-Sobolev assumption, with an iterated particle/time limit.  \citet[Theorems~6.3--6.4]{banerjee2026langevin} establish quantitative physical-time empirical KSD and $W_2$ propagation of chaos for fixed positive Langevin noise, under their target log-Sobolev, regularity, dissipativity and initialization assumptions; their bounds have the form $\sup_t\E d(\mu_t^N,\rho_t)$.  Our compact result uses deterministic dynamics conditional on initialization and the ordering $\E\sup_t$; the different assumptions and kernels preclude a direct rate comparison.

\begin{table}[t]
\centering
\small
\caption{Selected particle comparisons.  Physical time refers to the empirical measure at time $t$.  Each statement is subject to the hypotheses of the cited theorem.}
\label{tab:comparison}
\begin{tabular}{@{}L{0.23\textwidth}L{0.28\textwidth}L{0.41\textwidth}@{}}
\toprule
Work & Dynamics and initialization & Particle conclusion \tabularnewline
\midrule
\citet{liu2023gaussian} & Bilinear Gaussian SVGD; Gaussian initialization & Full physical-time $W_2$ control; $\sup_t\E$ \tabularnewline
\citet{balasubramanian2026uniform} & General classes; Gaussian bilinear special case & Broad metrics for time averages; Gaussian full $W_2$ with $\E\sup_t$ \tabularnewline
\citet{banerjee2026langevin} & Fixed positive Langevin noise & Physical-time empirical KSD and $W_2$; $\sup_t\E$ \tabularnewline
\citet{chen2025drmmd} & Adaptive DrMMD; target samples & Terminal-time $W_2$; sample requirements depend on horizon \tabularnewline
\citet{rosenzweig2026energy} & Energy-kernel flow; target-prepared data for all-time result & Fixed-horizon mean-field control; all-time target consistency for prepared data \tabularnewline
Present paper & Deterministic Green--Bessel flow; fixed smooth initial law on a compact manifold & Full physical-time MMD and $W_1$ control; $\E\sup_t$; exact-mode approximation \tabularnewline
\bottomrule
\end{tabular}
\end{table}

\subsection{Organization of the proof}

\Cref{sec:framework} states the compact and exact-mode results, followed by the all-time principle they use.  The construction in \Cref{sec:kernel-construction} produces a Stein velocity with the common energy; \Cref{sec:energy-population} then supplies population attraction, and \Cref{sec:relative-stability} proves the all-time implication.  \Cref{sec:compact-verification} verifies its compact-state inputs, while \Cref{sec:finite-truncation} controls the omitted spectral modes.  \Cref{sec:euclidean-targets,sec:gaussian-mixtures} develop the noncompact extension.  The appendices contain the measure-valued chain rules, spectral estimates, coupling and cutoff arguments, and the transfer to labelled marginals.

\section{Framework and main results}
\label{sec:framework}

We define the target-adapted dynamics and state the compact theorem and its exact spectral approximation.  We then formulate the all-time principle used in their proofs.  The conditional Euclidean extension is stated with its additional hypotheses in \Cref{sec:euclidean-targets}.

\subsection{State space, target, and weighted operators}

The state space $\X$ is either a compact connected smooth Riemannian manifold without boundary or $\R^d$.  Gradients, divergences, and distances are understood with respect to the Riemannian structure; in the Euclidean case they have their usual meaning.  The target probability measure is
\begin{equation}
\label{eq:target}
\pi(\dd x)=Z^{-1}e^{-V(x)}\,\dd x,
\qquad
Z=\int_{\X}e^{-V(x)}\,\dd x<\infty.
\end{equation}
On a manifold, $\dd x$ denotes Riemannian volume.

For a smooth vector field $F$, define the weighted divergence
\begin{equation}
\label{eq:weighted-div}
\divpi F
:=\pi^{-1}\nabla\cdot(\pi F)
=\nabla\cdot F+F\cdot\nabla\log\pi.
\end{equation}
Its defining integration-by-parts identity is
\begin{equation}
\label{eq:weighted-ibp}
\int_{\X}g\,\divpi F\,\dd\pi
=-\int_{\X}\nabla g\cdot F\,\dd\pi.
\end{equation}
The nonnegative Langevin generator is
\begin{equation}
\label{eq:A-def}
\A_\pi=-\divpi\nabla.
\end{equation}
It is understood as the self-adjoint operator associated with the closed Dirichlet form
\[
\mathcal E_\pi(f,g)=\int_{\X}\nabla f\cdot\nabla g\,\dd\pi.
\]
Constants lie in the null space.  We therefore work on
\[
L^2_0(\pi)=\left\{f\in L^2(\pi):\int f\,\dd\pi=0\right\}.
\]
The target has Poincar\'e gap $\lambda_1>0$ if
\begin{equation}
\label{eq:poincare}
\lambda_1\norm{f}_{L^2(\pi)}^2
\le \int\abs{\nabla f}^2\,\dd\pi,
\qquad f\in\Dom(\mathcal E_\pi)\cap L^2_0(\pi).
\end{equation}
Equivalently, $\A_\pi\ge\lambda_1\Id$ on $L^2_0(\pi)$.

Fix an integer $r\ge1$.  The Green--Bessel operator is
\begin{equation}
\label{eq:Q-def}
\Q_{r,\pi}
:=\A_\pi^{-1}(\Id+\A_\pi)^{-r}
\quad\text{on }L^2_0(\pi).
\end{equation}
When $\A_\pi$ has compact resolvent, choose an orthonormal eigenbasis
\begin{equation}
\label{eq:eigenpairs}
\A_\pi\varphi_n=\lambda_n\varphi_n,
\qquad
0=\lambda_0<\lambda_1\le\lambda_2\le\cdots,
\qquad
\varphi_0=1,
\end{equation}
and set
\begin{equation}
\label{eq:q-def}
q_r(\lambda)=\frac1{\lambda(1+\lambda)^r},
\qquad \lambda>0.
\end{equation}
Then
\begin{equation}
\label{eq:G-def}
\G_{r,\pi}(x,y)
=\sum_{n\ge1}q_r(\lambda_n)\varphi_n(x)\varphi_n(y)
\end{equation}
is the scalar kernel of $\Q_{r,\pi}$ whenever the series has the asserted regularity.

For a finite signed measure $\sigma$ for which the kernel integral exists, define its Green--Bessel potential by
\begin{equation}
\label{eq:Phi-def}
\PhiPot_\sigma(x)
:=\int_{\X}\G_{r,\pi}(x,y)\,\sigma(\dd y).
\end{equation}
If $\sigma=u\pi$ with $u\in L^2_0(\pi)$, then
\begin{equation}
\label{eq:compatibility}
\PhiPot_{u\pi}=\Q_{r,\pi}u.
\end{equation}
The kernel is centered in each variable, so $\PhiPot_\pi=0$ and $\PhiPot_\mu=\PhiPot_{\mu-\pi}$ for every probability measure in this domain.  Thus \eqref{eq:Phi-def} defines the potential on measures, while \eqref{eq:compatibility} relates it to the operator on densities.

\subsection{The Green--Bessel discrepancy and the SVGD dynamics}

Define the centered Hilbert space $\mathcal H_{r,\pi}=\Dom(\Q_{r,\pi}^{-1/2})$, with norm
\[
\norm{h}_{\mathcal H_{r,\pi}}=\norm{\Q_{r,\pi}^{-1/2}h}_{L^2(\pi)}.
\]
The kernel energy is finite for empirical measures after smoothing.  To discuss moment control on $\R^d$ before any integrability of the unknown measure is known, it is useful to define the same discrepancy through bounded test functions.  For a centered finite signed measure $\sigma$, set
\begin{equation}
\label{eq:D-def}
\D_{r,\pi}(\sigma)
=\sup\left\{|\sigma(\phi)|:\phi\in\mathcal T,
\ \norm{\phi-\pi(\phi)}_{\mathcal H_{r,\pi}}\le1\right\},
\end{equation}
where $\mathcal T=C^\infty(\X)$ on a compact manifold and $\mathcal T=C_c^\infty(\R^d)$ in the Euclidean case. The centered test functions are required to be dense in $\mathcal H_{r,\pi}$; this is verified for the target classes below. For probability measures, write $\D_{r,\pi}(\mu,\nu)=\D_{r,\pi}(\mu-\nu)$.  We write $\Pcal_1(\X)$ for the probability measures with finite first moment, the domain of $W_1$.

When point evaluation is continuous on $\mathcal H_{r,\pi}$ and
$\int \G_{r,\pi}(x,x)^{1/2}\,|\sigma|(\dd x)<\infty$, the Riesz representation theorem gives
\begin{equation}
\label{eq:D-spectral}
\D_{r,\pi}(\sigma)^2
=\iint\G_{r,\pi}(x,y)\,\sigma(\dd x)\sigma(\dd y)
=\sum_{n\ge1}q_r(\lambda_n)|\sigma(\varphi_n)|^2.
\end{equation}
All signed measures appearing in the compact dynamics satisfy this integrability condition. In the Euclidean statements it is imposed or verified explicitly. The dual definition \eqref{eq:D-def} also applies to probabilities for which no initial moment or spectral expansion has been assumed. It is jointly lower semicontinuous in $(\mu,\nu)$ under weak convergence, since it is a supremum of continuous test-function actions.

If $\mu=f\pi$ and $u=f-1\in L^2_0(\pi)$, then
\begin{equation}
\label{eq:D-density}
\D_{r,\pi}(\mu,\pi)^2=\ip{u}{\Q_{r,\pi}u}_{L^2(\pi)}.
\end{equation}
The discrepancy is the maximum mean discrepancy of $\G_{r,\pi}$ on its kernel-integrable domain \citep{sriperumbudur2010hilbert}.

The matrix-valued kernel used by SVGD is
\begin{equation}
\label{eq:K-def}
\K_{r,\pi}(x,y)
=
\sum_{n\ge1}
\frac{q_r(\lambda_n)}{\lambda_n}
\nabla\varphi_n(x)\otimes\nabla\varphi_n(y).
\end{equation}
On a manifold, $\K_{r,\pi}(x,y)$ maps $T_y\X$ to $T_x\X$.  The source-variable weighted divergence is taken row-wise, or covariantly in the manifold setting.  The construction in \Cref{sec:kernel-construction} proves
\begin{equation}
\label{eq:K-intertwine-framework}
\operatorname{div}_{\pi,y}\K_{r,\pi}(x,y)
=-\nabla_x\G_{r,\pi}(x,y).
\end{equation}
Hence the matrix-kernel SVGD velocity is
\begin{equation}
\label{eq:v-mu}
 v_\mu(x)
=\int_{\X}\operatorname{div}_{\pi,y}\K_{r,\pi}(x,y)\,\mu(\dd y)
=-\nabla\PhiPot_{\mu-\pi}(x).
\end{equation}
The centering of $\G_{r,\pi}$ against $\pi$ gives $\PhiPot_\pi=0$.

The particle system is
\begin{equation}
\label{eq:particle-system}
\dot X_i^N(t)
=-\frac1N\sum_{j=1}^N
\nabla_x\G_{r,\pi}(X_i^N(t),X_j^N(t)),
\qquad i=1,\ldots,N,
\end{equation}
with empirical measure
\begin{equation}
\label{eq:empirical}
\mu_t^N=\frac1N\sum_{i=1}^N\delta_{X_i^N(t)}.
\end{equation}
The population equation is
\begin{equation}
\label{eq:population}
\partial_t\rho_t+\nabla\cdot(\rho_t v_{\rho_t})=0,
\qquad
v_{\rho_t}=-\nabla\PhiPot_{\rho_t-\pi}.
\end{equation}
Both equations are instances of the same nonlinear continuity equation, with an atomic or absolutely continuous initial law.

\subsection{Quantitative compact theorem}

The main result concerns a smooth positive target on a compact connected manifold and independent particles drawn from a fixed smooth density.  Compactness makes the force bounds global; connectedness gives a positive target gap.  The smoothing threshold below makes point masses have finite energy and gives the derivative bounds needed for the particle comparison.  The initial density may vanish and need not be close to the target.

\begin{theorem}[Compact Green--Bessel SVGD]
\label{thm:compact-main}
Let $\X$ be a compact connected smooth Riemannian manifold without boundary, let $\pi$ have a smooth strictly positive density, and let $r\ge r_\star(d)$, where $r_\star(d)$ is the nonoptimal regularity threshold specified in \Cref{prop:compact-regularity}.  Let $\rho_0=f_0\pi$ with
\begin{equation}
\label{eq:compact-initial}
 f_0\in C^\infty(\X),
\qquad
0\le f_0\le M_0,
\qquad
\int f_0\,\dd\pi=1.
\end{equation}
Initialize \eqref{eq:particle-system} independently with common law $\rho_0$, and let $\rho_t$ solve \eqref{eq:population} with $\rho_{t=0}=\rho_0$.  The particle system and the classical population flow are globally well posed.  There is a constant $C$, depending on $(\X,\pi,r,f_0)$ but not on $N$ or time, such that
\begin{align}
\label{eq:compact-D-rate}
\E\sup_{t\ge0}\D_{r,\pi}(\mu_t^N,\rho_t)
&\le \frac{C}{\sqrt{\log(2+N)}},\\
\label{eq:compact-W-rate}
\E\sup_{t\ge0}\Wone(\mu_t^N,\rho_t)
&\le C[\log(2+N)]^{-1/[2(r+2)]}.
\end{align}
The population obeys
\begin{equation}
\label{eq:compact-pop-W-rate}
\Wone(\rho_t,\pi)
\le C(1+t)^{-1/[2(r+2)]}.
\end{equation}
Therefore, for every sequence $t_N\to\infty$,
\begin{equation}
\label{eq:compact-last-rate}
\E\Wone(\mu_{t_N}^N,\pi)
\le
C\left(
[\log(2+N)]^{-1/2}+t_N^{-1/2}
\right)^{1/(r+2)}.
\end{equation}
For every fixed $\ell\in\N$, the same last-iterate conclusion holds for the $\ell$-particle labelled marginal in $W_1$ for the additive product metric, with the $O(\ell^2/N)$ sampling-without-replacement correction given in \Cref{prop:labelled-comparison}.
\end{theorem}

\begin{remark}[Why the rate is logarithmic]
The squared finite-time comparison is bounded by $Ce^{CT}/N$, while the population discrepancy is $O(T^{-1/2})$. Choosing $T$ proportional to $\log N$, with a sufficiently small coefficient, yields \eqref{eq:compact-D-rate}.  This nonoptimal rate controls every physical time in a discrepancy that determines the full empirical measure.
\end{remark}

\subsection{Finite spectral approximation}

A finite implementation retains only part of the target spectrum.  Exact truncation preserves the Stein identity and dissipation laws; the price is a spatial-resolution error from the omitted modes.  The next theorem quantifies that error with constants uniform in the cutoff.

Let $\chi\in C_c^\infty([0,\infty))$ satisfy $0\le\chi\le1$, $\chi\equiv1$ on $[0,1]$, and $\chi\equiv0$ on $[2,\infty)$.  For $\Lambda\ge\lambda_1$, define
\begin{equation}
\label{eq:qLambda}
q_{r,\Lambda}(\lambda)
=q_r(\lambda)\chi(\lambda/\Lambda),
\end{equation}
with associated finite-rank scalar and matrix kernels $\G_{r,\pi}^{\Lambda}$ and $\K_{r,\pi}^{\Lambda}$.  Let $\mu_t^{N,\Lambda}$ and $\rho_t^{\Lambda}$ denote their particle and population flows, and let $\D_{r,\pi}^{\Lambda}$ be the corresponding finite-mode discrepancy.

\begin{theorem}[Exact finite-mode approximation]
\label{thm:finite-mode-main}
Under the assumptions of \Cref{thm:compact-main}, there is $C<\infty$, independent of $N$, $t$, and $\Lambda\ge\lambda_1$, such that
\begin{equation}
\label{eq:finite-mode-alltime-D}
\E\sup_{t\ge0}
\D_{r,\pi}^{\Lambda}(\mu_t^{N,\Lambda},\rho_t^{\Lambda})
\le C[\log(2+N)]^{-1/2}.
\end{equation}
For every sequence $t_N\to\infty$ and every choice $\Lambda_N\ge\lambda_1$,
\begin{equation}
\label{eq:finite-mode-last-bound}
\E\Wone(\mu_{t_N}^{N,\Lambda_N},\pi)
\le C\left[
\Lambda_N^{(r+1)/2}
\left(
[\log(2+N)]^{-1/2}+t_N^{-1/2}
\right)
+\Lambda_N^{-1/2}
\right].
\end{equation}
In particular, with
\begin{equation}
\label{eq:optimal-Lambda}
a_N=[\log(2+N)]^{-1/2}+t_N^{-1/2},
\qquad
\Lambda_N\asymp a_N^{-2/(r+2)},
\end{equation}
one has
\begin{equation}
\label{eq:finite-mode-optimized}
\E\Wone(\mu_{t_N}^{N,\Lambda_N},\pi)
\le C a_N^{1/(r+2)}.
\end{equation}
If $L_N$ is the number of retained modes, Weyl counting gives
\begin{equation}
\label{eq:mode-count}
L_N\lesssim \Lambda_N^{d/2}
\lesssim a_N^{-d/(r+2)}.
\end{equation}
After the values $\varphi_n(X_i)$ and $\nabla\varphi_n(X_i)$ are available, one evaluation of the finite-mode velocity costs $O(NL_Nd)$ arithmetic operations.
\end{theorem}

\begin{remark}[Exact modes versus numerical modes]
\Cref{thm:finite-mode-main} truncates exact eigenpairs of $\A_\pi$.  Approximate eigenpairs introduce a generator residual and hence a residual force, computed in \eqref{eq:approx-intertwining-residual}.  Such an error can accumulate over arbitrarily long times.  A numerical extension must therefore control generator residuals in a norm strong enough to evaluate the force on Dirac masses; pointwise kernel approximation alone is insufficient.
\end{remark}

\subsection{The all-time principle behind the compact theorem}

The compact proof combines population attraction, finite-time particle approximation, and monotonicity of a common target energy.  The following package records the intermediate facts needed for that argument and its extension to $W_1$.  Every item is established under the geometric and smoothing conditions of \Cref{thm:compact-main}; together they also specify what a noncompact application must verify.  The conditions concern a fixed initial law $\rho_0=f_0\pi$ and its i.i.d. particle initializations.

\begin{assumption}[Inputs for the all-time principle]
\label{ass:GB}
\leavevmode
\begin{enumerate}[label=\textup{(GB\arabic*)}]
\item \label{ass:GB-gap}
\emph{Population attraction.} The self-adjoint generator has compact resolvent, a smooth centered eigenbasis, and Poincar\'e gap $\lambda_1>0$.
\item \label{ass:GB-kernel}
\emph{Energy on measures.} The centered tests in \eqref{eq:D-def} are dense in $\mathcal H_{r,\pi}$. Point and first-derivative evaluations are continuous on this space. The scalar kernel is symmetric, locally $C^2$, and centered against $\pi$; its matrix lift and differentiated series satisfy \eqref{eq:K-intertwine-framework}. The measures used below satisfy the integrability condition for \eqref{eq:D-spectral}.
\item \label{ass:GB-pop}
\emph{Flow and dissipation.} The population and empirical flows exist globally and uniquely. The population has a bounded classical density ratio on every finite interval. The differentiations and integration by parts in the common-energy identity \eqref{eq:common-energy} and the entropy inequality \eqref{eq:entropy-controls-D} are valid along these flows.
\item \label{ass:GB-form}
\emph{Comparison on a finite interval.} For each fixed $T<\infty$, the particle--population comparison satisfies
\begin{equation}
\label{eq:abstract-form}
\E\sup_{0\le t\le T}\D_{r,\pi}(\mu_t^N,\rho_t)^2\le C_T/N,
\end{equation}
where $C_T$ is independent of $N$. This is a finite-horizon input, with no required control on its growth in $T$.
\item \label{ass:GB-compactness}
\emph{Passage to $W_1$.} The discrepancy separates probabilities, and its target-energy sublevels are relatively compact in $W_1$.
\item \label{ass:GB-moment}
\emph{Control of random energy sublevels.} For a fixed base point $x_\ast$, there is $C<\infty$ such that every probability measure of finite discrepancy satisfies
\begin{equation}
\label{eq:abstract-moment}
\int d(x,x_\ast)^2\,\mu(\dd x)\le C\bigl(1+\D_{r,\pi}(\mu,\pi)\bigr).
\end{equation}
This condition is automatic on compact $\X$.
\end{enumerate}
\end{assumption}

The gap turns population entropy dissipation into decay of the target discrepancy.  The finite-time estimate reaches a time when that discrepancy is small; common-energy monotonicity then controls the entire remaining trajectory.  Separation and compactness convert this control to $W_1$, while the moment bound allows an expectation over the random initial energy.  These roles are assembled in \Cref{sec:relative-stability}.

\begin{theorem}[From finite-time comparison to all-time convergence]
\label{thm:abstract-main}
Assume \Cref{ass:GB}.  Let $\rho_0=f_0\pi$, where $f_0$ is smooth, bounded, nonnegative, and normalized, and assume
\begin{equation}
\label{eq:initial-diagonal}
\int_{\X}\G_{r,\pi}(x,x)\,\rho_0(\dd x)<\infty.
\end{equation}
Initialize the particles independently with common law $\rho_0$.  Then the population satisfies
\begin{equation}
\label{eq:abstract-population-rate}
\D_{r,\pi}(\rho_t,\pi)^2
\le
\frac{\KL(\rho_0\mid\pi)}{\lambda_1t},
\qquad t>0.
\end{equation}
Moreover,
\begin{align}
\label{eq:abstract-alltime-D}
\lim_{N\to\infty}
\E\sup_{t\ge0}\D_{r,\pi}(\mu_t^N,\rho_t)&=0,\\
\label{eq:abstract-alltime-W1}
\lim_{N\to\infty}
\E\sup_{t\ge0}\Wone(\mu_t^N,\rho_t)&=0.
\end{align}
Consequently, for every deterministic sequence $t_N\to\infty$,
\begin{equation}
\label{eq:abstract-last}
\E\Wone(\mu_{t_N}^N,\pi)\longrightarrow0.
\end{equation}
For every fixed $\ell\in\N$, the labelled marginal satisfies
\begin{equation}
\label{eq:abstract-labelled}
\Law(X_1^N(t_N),\ldots,X_\ell^N(t_N))
\Longrightarrow \pi^{\otimes\ell}.
\end{equation}
\end{theorem}

\subsection{Extension to unbounded state spaces}

The assumptions in the all-time principle distinguish the part of the argument that uses compactness.  On $\R^d$, the smooth quadratic-type confinement in \Cref{ass:euclidean-target} gives local kernels and a moment bound; finite-time population regularity and transport stability require the additional input in \Cref{ass:euclidean-transport}.  \Cref{thm:euclidean-main} states the proved static results and the conditional dynamical extension together with its precise hypotheses in \Cref{sec:euclidean-targets}.  \Cref{cor:mixture-main-summary} treats common-covariance Gaussian mixtures, with an explicit Poincar\'e bound proved in \Cref{sec:gaussian-mixtures}.

\section{From the target generator to a matrix Stein kernel}
\label{sec:kernel-construction}

The all-time argument needs a velocity that dissipates the chosen target MMD.  For the Green--Bessel energy this velocity is $-\nabla\PhiPot_{\mu-\pi}$.  We construct a positive matrix kernel whose source Stein divergence gives exactly that force.  The eigenvalue equation determines the coefficient of each matrix feature, so positivity and the Stein identity can be checked in the same basis.

\subsection{The scalar Green--Bessel kernel}

Assume for the moment that $\A_\pi$ has compact resolvent and that the eigenfunction series below converges with the derivatives used in the argument.  From \eqref{eq:Q-def}--\eqref{eq:q-def},
\begin{equation}
\label{eq:Q-spectral}
\Q_{r,\pi}u
=\sum_{n\ge1}q_r(\lambda_n)
\ip{u}{\varphi_n}_{L^2(\pi)}\varphi_n,
\qquad u\in L^2_0(\pi).
\end{equation}
The scalar kernel is therefore \eqref{eq:G-def}.  Since every $\varphi_n$, $n\ge1$, is centered,
\begin{equation}
\label{eq:G-centered}
\int_{\X}\G_{r,\pi}(x,y)\,\pi(\dd y)=0,
\qquad
\int_{\X}\G_{r,\pi}(x,y)\,\pi(\dd x)=0.
\end{equation}
In particular,
\begin{equation}
\label{eq:Phi-pi-zero}
\PhiPot_\pi=0.
\end{equation}

The kernel is positive semidefinite because $q_r(\lambda_n)>0$.  For points $x_1,\ldots,x_m$ and coefficients $a_1,\ldots,a_m$,
\begin{equation}
\label{eq:G-positive}
\sum_{i,j=1}^m a_i a_j\G_{r,\pi}(x_i,x_j)
=
\sum_{n\ge1}q_r(\lambda_n)
\abs{\sum_{i=1}^m a_i\varphi_n(x_i)}^2
\ge0.
\end{equation}
The same expansion gives \eqref{eq:D-spectral}.  If $\G_{r,\pi}(x,x)<\infty$, then every empirical measure has finite energy, since
\begin{equation}
\label{eq:empirical-energy-finite}
\iint\G_{r,\pi}(x,y)\,\mu^N(\dd x)\mu^N(\dd y)
=\frac1{N^2}\sum_{i,j=1}^N\G_{r,\pi}(X_i,X_j).
\end{equation}
The diagonal terms $i=j$ are precisely where the Bessel smoothing is used.

\subsection{Why the scalar kernel is not enough}

For a scalar kernel $k$, classical SVGD uses the vector-valued kernel $k\Id$ and obtains
\begin{equation}
\label{eq:scalar-SVGD}
 v_\mu(x)
=\int_{\X}
\left[
\nabla_y k(x,y)+k(x,y)\nabla\log\pi(y)
\right]\mu(\dd y)
\end{equation}
in Euclidean coordinates.  Equivalently, the integrand is the weighted source divergence $\operatorname{div}_{\pi,y}(k(x,y)\Id)$.  Setting $k=\G_{r,\pi}$ does not generally give $-\nabla_x\PhiPot_{\mu-\pi}$.  The weighted source divergence in \eqref{eq:scalar-SVGD} and the negative gradient in the first variable are different operations for a general target-adapted kernel.

A matrix lift lets the source divergence act on an eigenfunction gradient, where the eigenvalue equation computes it exactly.  This is how the prescribed energy force becomes a Stein velocity.

\subsection{Positive matrix lift}

For tangent vectors $a\in T_x\X$ and $b\in T_y\X$, use the convention
\[
(a\otimes b)z=a\,\ip{b}{z}_{T_y\X}.
\]
Define
\begin{equation}
\label{eq:K-construction}
\K_{r,\pi}(x,y)
=
\sum_{n\ge1}
\frac{q_r(\lambda_n)}{\lambda_n}
\nabla\varphi_n(x)\otimes\nabla\varphi_n(y).
\end{equation}
The coefficient $q_r(\lambda_n)/\lambda_n$ is chosen because one weighted divergence in $y$ contributes the factor $-\lambda_n$.

\begin{proposition}[Positivity and symmetry]
\label{prop:K-positive}
Whenever the series \eqref{eq:K-construction} converges pointwise, $\K_{r,\pi}$ is a positive semidefinite bundle-valued kernel and
\begin{equation}
\label{eq:K-symmetry}
\K_{r,\pi}(y,x)=\K_{r,\pi}(x,y)^*.
\end{equation}
\end{proposition}

\begin{proof}
Let $x_1,\ldots,x_m\in\X$ and $c_i\in T_{x_i}\X$.  Then
\begin{align}
\sum_{i,j=1}^m
\ip{c_i}{\K_{r,\pi}(x_i,x_j)c_j}
&=
\sum_{n\ge1}\frac{q_r(\lambda_n)}{\lambda_n}
\sum_{i,j=1}^m
\ip{c_i}{\nabla\varphi_n(x_i)}
\ip{\nabla\varphi_n(x_j)}{c_j}
\\
&=
\sum_{n\ge1}\frac{q_r(\lambda_n)}{\lambda_n}
\abs{\sum_{i=1}^m\ip{c_i}{\nabla\varphi_n(x_i)}}^2
\ge0.
\end{align}
The adjoint symmetry follows by exchanging $x$ and $y$ in each rank-one term.
\end{proof}

The kernel therefore defines a vector-valued reproducing kernel Hilbert space in the standard sense; see, for example, \citet{micchelli2005learning}.  Its use in the SVGD variational construction follows the matrix-kernel framework of \citet{wang2019matrix}.  The representer identity identifies its Stein direction with the velocity computed below.

\subsection{The Stein--Green intertwining identity}

Each source divergence cancels the denominator $\lambda_n$ in the matrix coefficient.  Since
\begin{equation}
\label{eq:eigen-div}
\divpi\nabla\varphi_n=-\A_\pi\varphi_n=-\lambda_n\varphi_n,
\end{equation}
we have, in the source variable,
\begin{align}
\operatorname{div}_{\pi,y}
\left(
\nabla\varphi_n(x)\otimes\nabla\varphi_n(y)
\right)
&=
\nabla\varphi_n(x)\,
\divpi\nabla\varphi_n(y)
\\
&=-\lambda_n\varphi_n(y)\nabla\varphi_n(x).
\end{align}
Multiplying by $q_r(\lambda_n)/\lambda_n$ and summing gives the required force.

\begin{proposition}[Stein--Green intertwining]
\label{prop:intertwining}
Assume that the series \eqref{eq:G-def} and \eqref{eq:K-construction} can be differentiated termwise in the indicated variables.  Then
\begin{equation}
\label{eq:intertwining}
\operatorname{div}_{\pi,y}\K_{r,\pi}(x,y)
=-\nabla_x\G_{r,\pi}(x,y).
\end{equation}
Consequently, the matrix-kernel Stein velocity is
\begin{equation}
\label{eq:velocity-potential}
 v_\mu(x)
=-\int_{\X}\nabla_x\G_{r,\pi}(x,y)\,\mu(\dd y)
=-\nabla\PhiPot_{\mu-\pi}(x).
\end{equation}
\end{proposition}

\begin{proof}
The first statement is the termwise calculation above.  Integrating \eqref{eq:intertwining} against $\mu$ gives the first equality in \eqref{eq:velocity-potential}.  By \eqref{eq:G-centered}, the same integral against $\pi$ vanishes, which gives the second equality.
\end{proof}

\begin{remark}[Target score and normalization]
The construction uses the target generator and target score; it requires no estimate of the score of the evolving law.  Multiplying the unnormalized density $e^{-V}$ by a positive constant leaves the probability measure $\pi$, its score, $\A_\pi$, and the kernels normalized in $L^2(\pi)$ unchanged.  If instead one normalizes eigenfunctions in $L^2(e^{-V}\dd x)$, the scalar and matrix kernels differ from their probability-normalized versions by a common positive factor.  This changes only the deterministic time scale of the particle trajectories.  Thus the force can be specified up to a constant time scale without knowing $Z$; computing the target-adapted spectral features remains a separate numerical task.
\end{remark}

\begin{remark}[MMD energy and Stein discrepancy]
The velocity in \eqref{eq:velocity-potential} is the formal Wasserstein gradient velocity of $\frac12\D_{r,\pi}(\mu,\pi)^2$, an MMD energy; this is the gradient-flow construction studied by \citet{arbel2019mmd}.  It should be distinguished from gradient descent on a kernel Stein discrepancy as in \citet{korba2021ksd}.  For the matrix kernel used here, applying weighted divergence in both variables gives the scalar Stein kernel
\[
 k_{\mathrm{Stein}}(x,y)
 =\sum_{n\ge1}\lambda_n q_r(\lambda_n)\varphi_n(x)\varphi_n(y).
\]
Consequently, for $\rho=f\pi$ with $f-1\in L^2_0(\pi)$, its squared matrix-kernel Stein discrepancy is
\[
 \operatorname{KSD}_{\K_{r,\pi}}(\rho,\pi)^2
 =\sum_{n\ge1}(1+\lambda_n)^{-r}|(f-1)_n|^2.
\]
This is the population entropy dissipation in \Cref{sec:energy-population}.  The common target MMD energy has the different weights $q_r(\lambda_n)$.
\end{remark}

\subsection{Particle equation and feature factorization}

Substituting the empirical measure into \eqref{eq:velocity-potential} gives
\begin{equation}
\label{eq:particle-kernel-form}
\dot X_i^N
=-\frac1N\sum_{j=1}^N
\nabla_x\G_{r,\pi}(X_i^N,X_j^N).
\end{equation}
Define the empirical eigenfeature moments
\begin{equation}
\label{eq:feature-moment}
m_n^N(t)=\mu_t^N(\varphi_n)
=\frac1N\sum_{j=1}^N\varphi_n(X_j^N(t)).
\end{equation}
Then
\begin{equation}
\label{eq:particle-feature-form}
\dot X_i^N
=-\sum_{n\ge1}q_r(\lambda_n)m_n^N
\nabla\varphi_n(X_i^N).
\end{equation}
The target moments vanish: $\pi(\varphi_n)=0$ for $n\ge1$.  Formula \eqref{eq:particle-feature-form} makes the target adaptation explicit.  Low eigenmodes, which encode large-scale geometry of the target diffusion, receive the largest weights; high modes are suppressed by the Bessel factor.

The moments do not evolve independently.  Differentiating \eqref{eq:feature-moment} gives
\begin{equation}
\label{eq:feature-dynamics}
\frac{\dd}{\dd t}m_m^N
=-\sum_{n\ge1}q_r(\lambda_n)m_n^N
\left[
\frac1N\sum_{i=1}^N
\nabla\varphi_m(X_i^N)\cdot\nabla\varphi_n(X_i^N)
\right].
\end{equation}
Thus the feature vector is driven by a positive semidefinite empirical Gram matrix.  For a finite spectral cutoff, \eqref{eq:particle-feature-form} also removes the all-pairs kernel loop: one first computes the $L$ empirical moments and then evaluates a weighted feature-gradient sum at each particle.

\subsection{Two representative targets}

\begin{example}[Uniform target on the flat torus]
Let $\X=\mathbb T^d$ and let $\pi$ be uniform.  Then $\A_\pi=-\Delta$, the nonconstant eigenfunctions are Fourier modes, and
\[
q_r(4\pi^2\abs{k}^2)
=
\frac1{4\pi^2\abs{k}^2(1+4\pi^2\abs{k}^2)^r}.
\]
The construction is a smoothed inverse Laplacian.  Target adaptation is invisible only because the target geometry is translation invariant.
\end{example}

\begin{example}[Standard Gaussian target]
For $\pi=\mathcal N(0,\Id)$ on $\R^d$,
\[
\A_\pi=-\Delta+x\cdot\nabla.
\]
The eigenfunctions are multivariate Hermite polynomials and the eigenvalues are their total degrees.  A degree-$m$ Hermite moment is weighted by $[m(1+m)^r]^{-1}$.  In contrast with a translation-invariant kernel, both the features and their spatial gradients are adapted to Gaussian confinement.
\end{example}

\section{A common target energy and population relaxation}
\label{sec:energy-population}

Two dissipation laws drive the proof.  The quadratic discrepancy decreases along both particle and population trajectories, while population entropy decreases at a rate that controls this discrepancy.  Together they imply population relaxation and keep the particles close to the target once a finite-time comparison has brought them there.  We establish the two laws here; the comparison argument follows in \Cref{sec:relative-stability}.

\subsection{The Green--Bessel discrepancy as a weighted negative Sobolev norm}

The spectral weights make the discrepancy weak enough to include atoms and give an explicit relation to entropy dissipation.  For a centered signed measure $\sigma$ whose kernel mean exists, \eqref{eq:D-spectral} reads
\begin{equation}
\label{eq:D-spectral-section4}
\D_{r,\pi}(\sigma)^2
=
\sum_{n\ge1}q_r(\lambda_n)\abs{\sigma(\varphi_n)}^2.
\end{equation}
For a centered density $u\in L^2_0(\pi)$,
\begin{equation}
\label{eq:D-Q-form}
\D_{r,\pi}(u\pi)^2
=
\ip{u}{\Q_{r,\pi}u}_{L^2(\pi)}
=
\norm{\Q_{r,\pi}^{1/2}u}_{L^2(\pi)}^2.
\end{equation}
Since $\Q_{r,\pi}$ has spectral order $-2r-2$, this is a target-adapted negative Sobolev norm of order $r+1$.  Its inverse form is
\begin{equation}
\label{eq:inverse-form}
\ip{h}{\Q_{r,\pi}^{-1}h}_{L^2(\pi)}
=
\norm{\Q_{r,\pi}^{-1/2}h}_{L^2(\pi)}^2,
\qquad
\Q_{r,\pi}^{-1}=\A_\pi(\Id+\A_\pi)^r.
\end{equation}
Thus the potential $h=\Q_{r,\pi}u$ has inverse-form norm equal to the discrepancy of $u\pi$.  This reproducing Hilbert-space viewpoint will justify differentiation for atomic measures in \Cref{app:atomic-regularization} and express the transport condition in the Euclidean extension.

The dual definition separates probability measures through smooth test functions.  The spectral and kernel formulas hold under the integrability condition in \eqref{eq:D-spectral}, which the empirical and bounded-density population laws used here satisfy.  Converting this weak discrepancy to $W_1$ will require the compactness established in the target-specific sections.

\subsection{The common target-energy identity}

Let $\mu_t$ be either a smooth population solution or an empirical solution of
\begin{equation}
\label{eq:common-continuity}
\partial_t\mu_t+\nabla\cdot(\mu_t v_{\mu_t})=0,
\qquad
v_{\mu_t}=-\nabla\PhiPot_{\mu_t-\pi}.
\end{equation}
Set
\[
g_t=\PhiPot_{\mu_t-\pi}.
\]
Formally, symmetry of $\G_{r,\pi}$ gives
\[
\frac12\frac{\dd}{\dd t}\D_{r,\pi}(\mu_t,\pi)^2
=\langle\partial_t\mu_t,g_t\rangle.
\]
Testing the continuity equation with $g_t$ yields
\begin{equation}
\label{eq:target-energy-formal}
\frac12\frac{\dd}{\dd t}\D_{r,\pi}(\mu_t,\pi)^2
=
\int_{\X}\nabla g_t\cdot v_{\mu_t}\,\dd\mu_t
=-\int_{\X}\abs{\nabla g_t}^2\,\dd\mu_t.
\end{equation}
Because $g_t$ itself varies with the measure, this calculation requires a chain rule.  The kernel-mean proof in \Cref{app:atomic-regularization} applies to smooth laws and finite particle sums alike, including their diagonal self-interactions.  It gives the following identity.

\begin{proposition}[Common target-energy dissipation]
\label{prop:common-energy}
Under the kernel-mean differentiability conditions of \Cref{app:atomic-regularization}, classical population solutions and empirical solutions satisfy, in integrated form and almost everywhere in time,
\begin{equation}
\label{eq:common-energy}
\frac12\frac{\dd}{\dd t}\D_{r,\pi}(\mu_t,\pi)^2
=-\int_{\X}
\abs{\nabla\PhiPot_{\mu_t-\pi}(x)}^2\,\mu_t(\dd x)
\le0.
\end{equation}
In particular,
\begin{equation}
\label{eq:target-energy-monotone}
\D_{r,\pi}(\mu_t,\pi)
\le \D_{r,\pi}(\mu_s,\pi),
\qquad t\ge s\ge0.
\end{equation}
\end{proposition}

The right-hand side may vanish at configurations other than the target.  The identity nevertheless provides exactly the particle control needed later: after a comparison time $T$, the discrepancy from the target cannot increase.  Population attraction will come from the second dissipation law.

\subsection{The density-ratio equation}

For the population law, write
\begin{equation}
\label{eq:density-ratio}
\rho_t=f_t\pi,
\qquad
\nu_t=f_t-1.
\end{equation}
Here $\nu_t$ denotes the centered density error.  By \eqref{eq:compatibility},
\begin{equation}
\label{eq:population-potential-function}
\PhiPot_{\rho_t-\pi}=\Q_{r,\pi}\nu_t.
\end{equation}
The population velocity is therefore
\begin{equation}
\label{eq:population-v-density}
 v_t:=v_{\rho_t}=-\nabla\Q_{r,\pi}\nu_t.
\end{equation}
Dividing \eqref{eq:population} by $\pi$ and using \eqref{eq:weighted-div},
\begin{align}
0
&=\partial_t f_t+\pi^{-1}\nabla\cdot(\pi f_t v_t)\\
&=\partial_t f_t+v_t\cdot\nabla f_t+f_t\divpi v_t.
\end{align}
Moreover,
\begin{equation}
\label{eq:weighted-div-v}
\divpi v_t
=\divpi(-\nabla\Q_{r,\pi}\nu_t)
=\A_\pi\Q_{r,\pi}\nu_t
=(\Id+\A_\pi)^{-r}\nu_t.
\end{equation}
Hence
\begin{equation}
\label{eq:ratio-equation}
\partial_t f_t
+v_t\cdot\nabla f_t
+f_t(\Id+\A_\pi)^{-r}(f_t-1)=0.
\end{equation}

The resolvent power has the semigroup representation
\begin{equation}
\label{eq:resolvent-semigroup}
(\Id+\A_\pi)^{-r}
=\frac1{\Gamma(r)}\int_0^\infty s^{r-1}e^{-s}e^{-s\A_\pi}\,\dd s.
\end{equation}
It therefore preserves positivity and constants.  Along a characteristic $\dot X_t=v_t(X_t)$,
\begin{equation}
\label{eq:ratio-characteristic}
\frac{\dd}{\dd t}f_t(X_t)
=-f_t(X_t)\left[(\Id+\A_\pi)^{-r}f_t(X_t)-1\right].
\end{equation}
Consequently, nonnegativity is preserved and
\begin{equation}
\label{eq:ratio-bound}
0\le f_t\le e^t\norm{f_0}_{L^\infty(\pi)}
\end{equation}
on every finite interval.  Together with compact-state kernel regularity, this estimate gives classical population solutions on every finite interval.  Euclidean population existence and the required bounds are stated separately as hypotheses of the conditional extension.

\subsection{Entropy dissipation}

Population entropy supplies the missing attraction estimate.  The choice $\Q_{r,\pi}=\A_\pi^{-1}(\Id+\A_\pi)^{-r}$ makes its dissipation particularly simple: weighted integration by parts cancels the Green factor $\A_\pi^{-1}$.

Assume first that $f_t$ is strictly positive.  Since mass is conserved,
\begin{align}
\frac{\dd}{\dd t}\KL(\rho_t\mid\pi)
&=\int_{\X}\log f_t\,\partial_t f_t\,\dd\pi\\
&=\int_{\X}\nabla\log f_t\cdot v_t\,\dd\rho_t\\
&=-\int_{\X}\nabla f_t\cdot\nabla\Q_{r,\pi}\nu_t\,\dd\pi\\
&=-\ip{\nu_t}{\A_\pi\Q_{r,\pi}\nu_t}_{L^2(\pi)}\\
&=-\ip{\nu_t}{(\Id+\A_\pi)^{-r}\nu_t}_{L^2(\pi)}.
\end{align}
On compact state spaces, if $f_t$ may vanish, the same identity follows from regularized entropy and dominated convergence; see \Cref{app:entropy-chain-rule}.  For the Euclidean extension, the assumed finite-time bounds on the density and velocity justify the same calculation with spatial cutoffs.

\begin{proposition}[Population entropy dissipation]
\label{prop:entropy-dissipation}
For compact-state classical population solutions, and for Euclidean population solutions satisfying the stated regularity hypothesis,
\begin{equation}
\label{eq:entropy-dissipation}
\frac{\dd}{\dd t}\KL(\rho_t\mid\pi)
=-\ip{\nu_t}{(\Id+\A_\pi)^{-r}\nu_t}_{L^2(\pi)}.
\end{equation}
If $\pi$ has Poincar\'e gap $\lambda_1$, then
\begin{equation}
\label{eq:entropy-controls-D}
\frac{\dd}{\dd t}\KL(\rho_t\mid\pi)
\le-\lambda_1\D_{r,\pi}(\rho_t,\pi)^2.
\end{equation}
Consequently,
\begin{equation}
\label{eq:population-D-rate}
\D_{r,\pi}(\rho_t,\pi)^2
\le\frac{\KL(\rho_0\mid\pi)}{\lambda_1t},
\qquad t>0.
\end{equation}
\end{proposition}

\begin{proof}
The preceding chain rule identifies the entropy dissipation.  To compare it with the discrepancy, expand both quantities in the spectral basis:
\[
\ip{\nu}{(\Id+\A_\pi)^{-r}\nu}
=\sum_{n\ge1}(1+\lambda_n)^{-r}\abs{\nu_n}^2,
\]
whereas
\[
\D_{r,\pi}(\nu\pi)^2
=\sum_{n\ge1}\frac{\abs{\nu_n}^2}{\lambda_n(1+\lambda_n)^r}.
\]
Since $\lambda_n\ge\lambda_1$ for $n\ge1$, the first sum is at least $\lambda_1$ times the second.  Integrating \eqref{eq:entropy-controls-D} from $0$ to $t$ gives
\[
\lambda_1\int_0^t\D_{r,\pi}(\rho_s,\pi)^2\,\dd s
\le \KL(\rho_0\mid\pi).
\]
This controls the time average of the squared discrepancy.  Monotonicity from \Cref{prop:common-energy} turns it into a bound at the final time:
\[
t\D_{r,\pi}(\rho_t,\pi)^2
\le\int_0^t\D_{r,\pi}(\rho_s,\pi)^2\,\dd s.
\]
This proves \eqref{eq:population-D-rate}.
\end{proof}

\begin{remark}[The role of the spectral gap]
The gap is used precisely in \eqref{eq:entropy-controls-D}: it compares the two spectral weights and so converts entropy dissipation into discrepancy decay.  Kernel regularity justifies the identities; finite-time comparison and compactness will then turn this population estimate into the particle conclusions.  These are the separate roles of the corresponding inputs in \Cref{ass:GB}.
\end{remark}

\subsection{Exact i.i.d. fluctuation identities}

The initial sample contributes an error of order $N^{-1/2}$.  The same calculation also bounds the initial particle energy, which will keep the entire trajectory inside a controlled energy sublevel.

\begin{proposition}[Initial empirical fluctuation]
\label{prop:initial-fluctuation}
Let $X_1,\ldots,X_N$ be i.i.d. with law $\rho$, let $\mu^N=N^{-1}\sum_i\delta_{X_i}$, and assume $\int\G_{r,\pi}(x,x)\,\rho(\dd x)<\infty$.  Then
\begin{equation}
\label{eq:initial-relative-fluctuation}
\E\D_{r,\pi}(\mu^N,\rho)^2
=
\frac1N\left[
\int_{\X}\G_{r,\pi}(x,x)\,\rho(\dd x)
-\D_{r,\pi}(\rho,\pi)^2
\right].
\end{equation}
Moreover,
\begin{equation}
\label{eq:initial-target-energy}
\E\D_{r,\pi}(\mu^N,\pi)^2
=
\left(1-\frac1N\right)\D_{r,\pi}(\rho,\pi)^2
+\frac1N\int_{\X}\G_{r,\pi}(x,x)\,\rho(\dd x).
\end{equation}
\end{proposition}

\begin{proof}
Expand the quadratic energy and separate the diagonal terms $i=j$ from the off-diagonal terms $i\ne j$.  Independence gives
\[
\E\G_{r,\pi}(X_i,X_j)
=
\iint\G_{r,\pi}(x,y)\,\rho(\dd x)\rho(\dd y)
=\D_{r,\pi}(\rho,\pi)^2
\]
for $i\ne j$, where the final equality uses the centering \eqref{eq:G-centered}.  Subtracting the deterministic cross terms gives \eqref{eq:initial-relative-fluctuation}; retaining the target terms gives \eqref{eq:initial-target-energy}.
\end{proof}

In the Euclidean case, \eqref{eq:initial-target-energy} supplies the moment bound for the random energy sublevel used in the passage to $W_1$.

\section{Finite-time comparison and the all-time argument}
\label{sec:relative-stability}

We now combine the two dissipation laws with a finite-time particle comparison.  Choose a deterministic time $T$: up to $T$, compare the two flows directly; after $T$, compare each flow with the target.  Population relaxation makes the second comparison small, and target-energy monotonicity preserves it for every later time.  This division at a single time is the reason a finite-horizon estimate can yield an all-time conclusion.

\subsection{The finite-time input}
\label{sec:relative-form-pairing}

We use the finite-time comparison in \Cref{ass:GB}:
\begin{equation}
\label{eq:finite-horizon-pathwise}
\E\sup_{0\le t\le T}\D_{r,\pi}(\mu_t^N,\rho_t)^2
\le \frac{C_T}{N},\qquad T<\infty.
\end{equation}
The supremum inside the expectation controls the whole initial interval.  On a compact manifold, the coupling proof in \Cref{prop:compact-transport,app:finite-mode-mean-field} gives $C_T\le C_0e^{C_1T}$, uniformly in $N$ and $T$.  The qualitative argument needs only that $C_T$ be finite for each fixed $T$.

\subsection{The stitching lemma}

\begin{lemma}[Stitching]
\label{lem:stitching}
For every $T\ge0$ and $t\ge T$,
\begin{equation}
\label{eq:stitching}
\D_{r,\pi}(\mu_t^N,\rho_t)
\le \D_{r,\pi}(\mu_T^N,\rho_T)
+2\D_{r,\pi}(\rho_T,\pi).
\end{equation}
\end{lemma}

\begin{proof}
The triangle inequality and \Cref{prop:common-energy} give
\begin{align*}
\D_{r,\pi}(\mu_t^N,\rho_t)
&\le\D_{r,\pi}(\mu_t^N,\pi)+\D_{r,\pi}(\rho_t,\pi)\\
&\le\D_{r,\pi}(\mu_T^N,\pi)+\D_{r,\pi}(\rho_T,\pi)\\
&\le\D_{r,\pi}(\mu_T^N,\rho_T)+2\D_{r,\pi}(\rho_T,\pi).
\end{align*}
\end{proof}

The lemma uses the particle dynamics only through target-energy monotonicity.  It therefore applies even when the particles do not approach the target for a fixed $N$.

\subsection{Uniform-in-time convergence in discrepancy}

For $T>0$, \Cref{lem:stitching} and the population bound give
\begin{equation}
\label{eq:abstract-stitch-preopt}
\E\sup_{t\ge0}\D_{r,\pi}(\mu_t^N,\rho_t)
\le \sqrt{C_T/N}
+2\sqrt{\frac{\KL(\rho_0\mid\pi)}{\lambda_1T}}.
\end{equation}
Indeed, the bound after $T$ is at most the supremum on $[0,T]$ plus the population error displayed above.  Keep $T$ fixed while sending $N\to\infty$, so the finite-time error vanishes; then send $T\to\infty$ to remove the population error.  This proves \eqref{eq:abstract-alltime-D} without any growth assumption on $C_T$.

For the compact theorem, \Cref{prop:compact-transport} gives
\begin{equation}
\label{eq:quantitative-stitch-preopt}
\E\sup_{t\ge0}\D_{r,\pi}(\mu_t^N,\rho_t)
\le \frac{C}{\sqrt N}e^{CT}+CT^{-1/2}.
\end{equation}
The quantitative bound balances a comparison error that grows with $T$ against a population error that decays with $T$.  Choose $T=c\log(2+N)$ with a fixed sufficiently small $c>0$.  The first term then decays algebraically in $N$, and the second gives
\begin{equation}
\label{eq:quantitative-D-log}
\E\sup_{t\ge0}\D_{r,\pi}(\mu_t^N,\rho_t)
\le C[\log(2+N)]^{-1/2}.
\end{equation}
An enlargement of $C$ covers every $N\ge1$.

\subsection{From discrepancy to Wasserstein distance}

On a compact manifold, \Cref{prop:compact-interpolation} proves
\begin{equation}
\label{eq:D-to-W-compact}
\Wone(\mu,\nu)\le C\D_{r,\pi}(\mu,\nu)^{1/(r+2)}.
\end{equation}
Apply this inequality pathwise to the time supremum and then use Jensen's inequality.  Together with \eqref{eq:quantitative-D-log}, this yields \eqref{eq:compact-W-rate}.

For the abstract theorem, interpolation is replaced by compactness on energy sublevels.  This requires two steps: establish a uniform comparison on each fixed sublevel, then control the random level selected by the initial sample.  Let
\[
\mathcal E_R=\{\nu\in\Pcal_1(\X):\D_{r,\pi}(\nu,\pi)\le R\}.
\]
By the compactness, lower semicontinuity, and separation hypotheses in \Cref{ass:GB}, there is a bounded nondecreasing modulus $\omega_R$ with $\omega_R(a)\to0$ as $a\downarrow0$ such that
\begin{equation}
\label{eq:sublevel-modulus}
\Wone(\mu,\nu)\le\omega_R(\D_{r,\pi}(\mu,\nu)),
\qquad\mu,\nu\in\mathcal E_R.
\end{equation}
To prove this, otherwise choose two sequences in $\mathcal E_R$ whose discrepancies tend to zero while their $W_1$ distances stay positive.  Compactness gives convergent subsequences.  Lower semicontinuity puts their limits in $\mathcal E_R$ and makes their limiting discrepancy zero; separation makes the limits equal, a contradiction.  Boundedness of the modulus follows from compactness.

To apply this deterministic modulus to the random particle path, write $R_N=\D_{r,\pi}(\mu_0^N,\pi)$ and $R_0=\D_{r,\pi}(\rho_0,\pi)$.  The initial fluctuation estimate implies
\begin{equation}
\label{eq:random-energy-moment}
\sup_N\E R_N^2<\infty.
\end{equation}
Target-energy monotonicity puts both entire trajectories in $\mathcal E_R$ on $\{R_N\le R\}$ whenever $R\ge R_0$.  On this event \eqref{eq:sublevel-modulus}, \eqref{eq:abstract-alltime-D}, and bounded convergence in probability give convergence to zero of the expected time supremum of $W_1$.

It remains to control samples outside the chosen sublevel.  This is where the moment hypothesis in \Cref{ass:GB} enters: on an unbounded state space,
\[
\sup_{t\ge0}\Wone(\mu_t^N,\rho_t)^2\le C(1+R_N+R_0).
\]
Cauchy--Schwarz and \eqref{eq:random-energy-moment} therefore bound the expectation on $\{R_N>R\}$ by $C/R$, uniformly in $N$.  On a compact space the diameter bound gives the same conclusion directly.  Send $N\to\infty$ and then $R\to\infty$ to prove \eqref{eq:abstract-alltime-W1}.

\subsection{Last iterates and labelled marginals}

The uniform comparison allows the observation time to grow at any rate with $N$.  Indeed, the population stays in $\mathcal E_{R_0}$, so its discrepancy decay and \eqref{eq:sublevel-modulus} imply $\rho_t\to\pi$ in $W_1$.  For any deterministic $t_N\to\infty$,
\begin{equation}
\label{eq:last-iterate-triangle}
\Wone(\mu_{t_N}^N,\pi)
\le\Wone(\mu_{t_N}^N,\rho_{t_N})+\Wone(\rho_{t_N},\pi).
\end{equation}
Taking expectations proves \eqref{eq:abstract-last}.  The empirical-to-labelled comparison in \Cref{prop:labelled-comparison} proves \eqref{eq:abstract-labelled} and the corresponding statement uniformly in time with target $\rho_t$.

\begin{proof}[Proof of \Cref{thm:abstract-main}]
The admissibility inputs justify \Cref{prop:entropy-dissipation} and hence the population bound.  Stitching gives the all-time discrepancy limit, the energy-sublevel argument gives its $W_1$ counterpart, and \eqref{eq:last-iterate-triangle} gives convergence at every diverging observation time.  Finally, \Cref{prop:labelled-comparison} transfers the empirical conclusions to each fixed number of labelled particles.
\end{proof}

\section{Compact connected manifolds}
\label{sec:compact-verification}

On a compact connected manifold, the geometric assumptions supply every input to the all-time argument.  Smoothness and a sufficiently large $r$ give a regular interaction and well-posed flows; coupling gives the finite-time estimate; heat smoothing converts discrepancy to $W_1$.  Compactness supplies the moment control.  We verify these inputs in that order and then assemble \Cref{thm:compact-main}.

\subsection{A nonoptimal regularity threshold}

Let
\begin{equation}
\label{eq:md-rstar}
m_d=2d+4,
\qquad
r_\star(d)=m_d+\left\lceil\frac d2\right\rceil+3.
\end{equation}
The threshold is chosen to supply all the differentiated kernel bounds and population smoothing estimates with a single condition on $r$.  We use this common threshold for convenience and do not optimize it.

\begin{proposition}[Compact Green--Bessel regularity]
\label{prop:compact-regularity}
Let $\X$ be a compact connected smooth $d$-manifold without boundary, let $\pi$ have a smooth strictly positive density, and let $r\ge r_\star(d)$.  Then:
\begin{enumerate}[label=\textup{(\roman*)}]
\item $\A_\pi$ has compact resolvent and a positive spectral gap on $L^2_0(\pi)$;
\item $\Q_{r,\pi}$ is a classical pseudodifferential operator of order $-2r-2$ on the centered subspace;
\item $\G_{r,\pi}$ and $\K_{r,\pi}$ have the derivatives required in \Cref{ass:GB-kernel}, the force in \eqref{eq:particle-system} is globally Lipschitz, and the differentiated eigenfunction series converge uniformly;
\item for every integer $0\le j\le m_d$,
\begin{align}
\label{eq:compact-feature-sum-v}
\sup_{x\in\X}
\sum_{n\ge1}q_r(\lambda_n)
\abs{\nabla^{j+1}\varphi_n(x)}^2&<\infty,\\
\label{eq:compact-feature-sum-div}
\sup_{x\in\X}
\sum_{n\ge1}\lambda_n^2q_r(\lambda_n)
\abs{\nabla^{j}\varphi_n(x)}^2&<\infty.
\end{align}
\end{enumerate}
\end{proposition}

\begin{proof}
The operator $\A_\pi$ is uniformly elliptic with smooth coefficients.  Its Friedrichs realization is self-adjoint with compact resolvent.  Connectedness and \eqref{eq:weighted-ibp} imply that the zero eigenspace consists exactly of constants, so the first nonzero eigenvalue is positive.  Elliptic regularity makes the eigenfunctions smooth, and finite sums of centered eigenfunctions are dense in $\mathcal H_{r,\pi}$ by the spectral theorem.

By the functional calculus for elliptic operators, $\Q_{r,\pi}=\A_\pi^{-1}(\Id+\A_\pi)^{-r}$ is classical of order $-2r-2$ on the orthogonal complement of constants; see \citet{seeley1967complex,hormander1985analysis3}.  The differentiated local Weyl law \citep[Theorem~1.1, with $s=0$ and order $m=2$]{rivera2022weighted}, applied to $\Id+\A_\pi$ and then omitting the constant eigenfunction, gives for $L\ge\max(1,\lambda_1)$
\begin{equation}
\label{eq:local-weyl-derivative}
\sup_{x\in\X}
\sum_{\lambda_n\le L}
\abs{\nabla^k\varphi_n(x)}^2
\le C_k L^{d/2+k}.
\end{equation}
A dyadic summation of \eqref{eq:local-weyl-derivative}, using $q_r(\lambda)\asymp\lambda^{-r-1}$ at high frequency, proves \eqref{eq:compact-feature-sum-v} when $r>d/2+j$, and \eqref{eq:compact-feature-sum-div} when $r>d/2+j+1$.  The choice \eqref{eq:md-rstar} is stronger than both requirements for $j\le m_d$.

These summable derivative bounds give uniform convergence of the scalar and matrix kernel series, with derivatives split between $x$ and $y$ by Cauchy--Schwarz.  In particular, $\G_{r,\pi}$ has finite diagonal, $\nabla_x\G_{r,\pi}$ is continuously differentiable in both variables, and the particle force is globally Lipschitz.  They also justify the termwise divergence in \Cref{prop:intertwining}.  The tail calculation is recorded in \Cref{app:compact-analysis}.
\end{proof}

\subsection{Population well-posedness}

The interaction $b(x,y)=-\nabla_x\G_{r,\pi}(x,y)$ is bounded and continuously differentiable on $\X\times\X$.  In particular, for every probability measure $\mu$,
\begin{equation}
\label{eq:compact-velocity-uniform}
\norm{v_\mu}_{C^1}+\norm{\divpi v_\mu}_{L^\infty}\le C.
\end{equation}
The constants depend only on the compact geometry and kernel, and the same bounds hold uniformly for every truncation.  These bounds control the characteristic flow; the additional smoothing of $\Q_{r,\pi}$ propagates smoothness of the density.

\begin{proposition}[Compact population well-posedness]
\label{prop:compact-pop-wp}
Under the assumptions of \Cref{prop:compact-regularity}, every initial ratio satisfying \eqref{eq:compact-initial} generates a unique global classical solution $\rho_t=f_t\pi$ of \eqref{eq:population}.  Nonnegativity and mass are preserved.  The same statement holds for every exact spectral truncation.
\end{proposition}

\begin{proof}
First construct the flow at the level of measures.  The bounded Lipschitz interaction makes the characteristic fixed-point map a contraction on a sufficiently short interval: differences of characteristic flows are bounded by the time integral of their spatial difference and the $W_1$ difference of the proposed measure paths.  Iteration gives uniqueness on every finite interval, and compactness and the uniform velocity bound extend the solution globally.

It remains to show that smooth initial densities stay smooth on every finite interval.  The same construction gives a local smooth solution, whose relative density satisfies
\[
(\partial_t+v_{\rho_t}\cdot\nabla)f_t=-f_t\divpi v_{\rho_t}.
\]
Consequently $f_t\ge0$, mass is conserved, and
$\norm{f_t}_{L^\infty}\le e^{Ct}\norm{f_0}_{L^\infty}$ by \eqref{eq:compact-velocity-uniform}.  Higher derivatives cannot blow up at a finite time.  More explicitly, elliptic Sobolev regularity for $\Q_{r,\pi}$, followed by Sobolev embedding, gives for every integer $k\ge1$
\begin{equation}
\label{eq:compact-pop-smoothing}
\norm{v_{f\pi}}_{C^{k+1}}+
\norm{\divpi v_{f\pi}}_{C^k}
\le C_k\bigl(1+\norm{f}_{C^{k-1}}\bigr).
\end{equation}
For example, $\Q_{r,\pi}:H^{k-1}\to H^{k+2r+1}$ and
$\A_\pi\Q_{r,\pi}:H^{k-1}\to H^{k+2r-1}$, and the threshold \eqref{eq:md-rstar} leaves more than $d/2$ derivatives for the stated embeddings.  Differentiating the transport equation $k$ times along characteristics bounds the upper time derivative of $\norm{f_t}_{C^k}$ by
$C_k\norm{f_t}_{C^k}+P_k(\norm{f_t}_{C^{k-1}})$ for a fixed polynomial $P_k$ with nonnegative coefficients.  Here the highest derivative of $f_t$ is multiplied only by the uniformly bounded first velocity derivatives or by $\divpi v_{\rho_t}$; all other terms are controlled by \eqref{eq:compact-pop-smoothing}.  Induction and Gronwall control every $C^k$ norm on every finite interval, extending the smooth solution globally.  Uniqueness of the measure solution identifies these extensions.

For a truncated operator, $0\le\chi\le1$ preserves the Sobolev bounds, so the same argument applies uniformly in the cutoff.  The characteristic formula also shows why nonnegativity suffices: zeros are transported by the flow.
\end{proof}

\subsection{Finite-time comparison by coupling}

The coupling estimate controls distances between particle positions.  To turn that estimate into discrepancy control, realize the kernel through its feature map:
\begin{equation}
\label{eq:compact-feature-map}
\Psi(x)=\bigl(\sqrt{q_r(\lambda_n)}\varphi_n(x)\bigr)_{n\ge1}\in\ell^2.
\end{equation}
The derivative sums of \Cref{prop:compact-regularity} imply that $\Psi$ is bounded and continuously differentiable as an $\ell^2$-valued map, and
\[
\D_{r,\pi}(\mu,\nu)=\left\|\int\Psi\,\dd(\mu-\nu)\right\|_{\ell^2}.
\]

\begin{proposition}[Compact finite-time coupling]
\label{prop:compact-transport}
There are constants $C_0,C_1<\infty$, independent of $N$ and $T\ge0$, such that
\begin{equation}
\label{eq:compact-finite-time}
\E\sup_{0\le t\le T}\D_{r,\pi}(\mu_t^N,\rho_t)^2
\le\frac{C_0}{N}e^{C_1T}.
\end{equation}
\end{proposition}

\begin{proof}
Insert independent population characteristics started at the same positions as the interacting particles.  This splits the error into a coupling error and an i.i.d. sampling error.  The first is controlled by the Lipschitz interaction and then transferred to discrepancy through the Lipschitz feature map.  For the second, bounds on the feature paths and their time derivatives control the supremum over $[0,T]$.  The resulting squared errors are $Ce^{CT}/N$ and $C(1+T^2)/N$, respectively.  \Cref{app:finite-mode-mean-field} proves these estimates, including the diagonal self-interaction and the manifold distance comparison, simultaneously for the full kernel and every cutoff.
\end{proof}

\subsection{Separation of probability measures}

To see directly why retaining every nonconstant mode distinguishes measures, suppose $\D_{r,\pi}(\mu,\nu)=0$.  Then
\begin{equation}
\label{eq:eigenmoments-equal}
(\mu-\nu)(\varphi_n)=0,
\qquad n\ge1.
\end{equation}
To pass from eigenfunctions to arbitrary continuous functions, let $P_s=e^{-s\A_\pi}$ be the Langevin heat semigroup.  For $\phi\in C(\X)$ and $s>0$, $P_s\phi$ is smooth and its spectral expansion converges uniformly.  By \eqref{eq:eigenmoments-equal},
\[
(\mu-\nu)(P_s\phi)=0.
\]
The Feller property and compactness give $P_s\phi\to\phi$ uniformly as $s\downarrow0$, hence $(\mu-\nu)(\phi)=0$.  Therefore $\mu=\nu$.

\begin{proposition}[Characteristicness on compact manifolds]
\label{prop:compact-separation}
Under the assumptions of \Cref{prop:compact-regularity}, $\D_{r,\pi}$ is a metric on $\Pcal(\X)$.
\end{proposition}

\subsection{Heat-semigroup interpolation to \texorpdfstring{$W_1$}{W1}}

Heat smoothing makes a Lipschitz test function regular enough to pair with the discrepancy.  Its smoothing error is $O(\sqrt s)$, while its inverse-form norm grows as $s\downarrow0$.  Balancing these two costs gives the exponent $1/(r+2)$ in the main theorem.

\begin{proposition}[Compact interpolation]
\label{prop:compact-interpolation}
Under the assumptions of \Cref{prop:compact-regularity}, there is $C<\infty$ such that for all probability measures $\mu,\nu$,
\begin{equation}
\label{eq:compact-interpolation}
\Wone(\mu,\nu)
\le C\min\left\{1,
\D_{r,\pi}(\mu,\nu)^{1/(r+2)}
\right\}.
\end{equation}
\end{proposition}

\begin{proof}
Let $\sigma=\mu-\nu$ and let $\phi$ be $1$-Lipschitz, normalized by $\pi(\phi)=0$.  Split the test function into its heat-smoothed part and the approximation error: for $0<s\le1$,
\begin{equation}
\label{eq:heat-split}
\sigma(\phi)=\sigma(P_s\phi)+\sigma(\phi-P_s\phi).
\end{equation}
The diffusion moves an expected distance of order $\sqrt s$, which gives
\begin{equation}
\label{eq:heat-lip-approx}
\norm{\phi-P_s\phi}_{L^\infty}
\le C\sqrt s,
\end{equation}
as proved in \Cref{app:compact-analysis}.  For the smoothed term, spectral Cauchy--Schwarz gives
\begin{align}
\abs{\sigma(P_s\phi)}
&\le
\D_{r,\pi}(\sigma)
\left[
\sum_{n\ge1}q_r(\lambda_n)^{-1}
e^{-2s\lambda_n}
\abs{\ip{\phi}{\varphi_n}_{L^2(\pi)}}^2
\right]^{1/2}\\
&\le
C s^{-(r+1)/2}
\D_{r,\pi}(\sigma)\norm{\phi}_{L^2(\pi)}.
\end{align}
The normalization and compactness bound $\norm{\phi}_{L^2(\pi)}$ uniformly over $1$-Lipschitz $\phi$.  Hence
\begin{equation}
\label{eq:heat-intermediate}
\Wone(\mu,\nu)
\le C\left[s^{-(r+1)/2}\D_{r,\pi}(\mu,\nu)+\sqrt s\right].
\end{equation}
For small discrepancy, choose $s=\D_{r,\pi}(\mu,\nu)^{2/(r+2)}$.  For large discrepancy, use the diameter bound for $W_1$.
\end{proof}

\subsection{Proof of the compact theorem}

\begin{proof}[Proof of \Cref{thm:compact-main}]
\Cref{prop:compact-regularity,prop:compact-pop-wp,prop:compact-separation} give the kernel, well-posedness, gap, and separation properties.  Continuity of the kernel makes its discrepancy continuous under weak convergence of probability measures, and $\Pcal(\X)$ is compact in $W_1$; moment control is automatic.  The energy and entropy identities are justified in \Cref{sec:energy-population,app:atomic-regularization}.  Together with \Cref{prop:compact-transport}, these verify \Cref{ass:GB}.  Thus \eqref{eq:quantitative-D-log} gives \eqref{eq:compact-D-rate}.  Applying \Cref{prop:compact-interpolation} and Jensen's inequality gives \eqref{eq:compact-W-rate}.

For the population, combine \eqref{eq:population-D-rate} with \eqref{eq:compact-interpolation} to obtain \eqref{eq:compact-pop-W-rate}.  Finally, use
\[
\Wone(\mu_{t_N}^N,\pi)
\le\Wone(\mu_{t_N}^N,\rho_{t_N})
+\Wone(\rho_{t_N},\pi)
\]
and the two preceding bounds.  Since $a\mapsto a^{1/(r+2)}$ is concave, the sum can be written in the form \eqref{eq:compact-last-rate}.  The labelled-marginal statement follows from \Cref{app:labelled-marginals}.
\end{proof}

\section{Finite spectral approximation}
\label{sec:finite-truncation}

An implementation uses finitely many target features.  Exact spectral truncation preserves the Stein identity, both dissipation laws, and the finite-time coupling estimate with constants independent of the cutoff.  Thus the all-time argument still controls the retained features.  The new issue is spatial resolution: measures with identical retained moments can differ at smaller scales.  A final interpolation estimate quantifies that error and proves \Cref{thm:finite-mode-main}.

\subsection{Exact truncated kernels and dynamics}

Fix the cutoff $\chi$ from \eqref{eq:qLambda}.  Define
\begin{align}
\label{eq:G-Lambda-def}
\G_{r,\pi}^{\Lambda}(x,y)
&=\sum_{n\ge1}q_r(\lambda_n)\chi(\lambda_n/\Lambda)
\varphi_n(x)\varphi_n(y),\\
\label{eq:K-Lambda-def}
\K_{r,\pi}^{\Lambda}(x,y)
&=\sum_{n\ge1}
\frac{q_r(\lambda_n)\chi(\lambda_n/\Lambda)}{\lambda_n}
\nabla\varphi_n(x)\otimes\nabla\varphi_n(y).
\end{align}
Only eigenvalues below $2\Lambda$ occur.  For a centered signed measure $\sigma$, write
\begin{equation}
\label{eq:Phi-Lambda-def}
\PhiPot^{\Lambda}_{\sigma}(x)
=\int_{\X}\G_{r,\pi}^{\Lambda}(x,y)\,\sigma(\dd y)
\end{equation}
and
\begin{equation}
\label{eq:D-Lambda-def}
\D_{r,\pi}^{\Lambda}(\sigma)^2
=\sum_{n\ge1}q_r(\lambda_n)\chi(\lambda_n/\Lambda)
\abs{\sigma(\varphi_n)}^2.
\end{equation}

Because the sums are finite, the weighted-divergence calculation is exact term by term:
\begin{equation}
\label{eq:truncated-intertwining}
\operatorname{div}_{\pi,y}\K_{r,\pi}^{\Lambda}(x,y)
=-\nabla_x\G_{r,\pi}^{\Lambda}(x,y).
\end{equation}
Thus the truncated SVGD velocity is
\begin{equation}
\label{eq:truncated-velocity}
v_{\mu}^{\Lambda}(x)
=-\nabla\PhiPot^{\Lambda}_{\mu-\pi}(x)
=-\int_{\X}\nabla_x\G_{r,\pi}^{\Lambda}(x,y)\,\mu(\dd y).
\end{equation}
The particle and population equations are
\begin{align}
\label{eq:truncated-particles}
\dot X_i^{N,\Lambda}(t)
&=v_{\mu_t^{N,\Lambda}}^{\Lambda}
  (X_i^{N,\Lambda}(t)),
&
\mu_t^{N,\Lambda}
&=\frac1N\sum_{i=1}^N\delta_{X_i^{N,\Lambda}(t)},\\
\label{eq:truncated-population}
\partial_t\rho_t^{\Lambda}
+\nabla\cdot(\rho_t^{\Lambda}v_{\rho_t^{\Lambda}}^{\Lambda})
&=0.
\end{align}
The target is stationary because every retained eigenfunction is centered against $\pi$.

\begin{proposition}[Structure preserved by exact truncation]
\label{prop:truncated-structure}
Under the assumptions of \Cref{thm:compact-main}, both solutions of
\eqref{eq:truncated-particles}--\eqref{eq:truncated-population} are global and satisfy
\begin{equation}
\label{eq:truncated-target-energy}
\frac12\frac{\dd}{\dd t}
\D_{r,\pi}^{\Lambda}(\mu_t,\pi)^2
=-\int_{\X}
\abs{\nabla\PhiPot^{\Lambda}_{\mu_t-\pi}}^2\,\dd\mu_t
\le0.
\end{equation}
If $\rho_t^{\Lambda}=f_t^{\Lambda}\pi$ and
$\nu_t^{\Lambda}=f_t^{\Lambda}-1$, then
\begin{equation}
\label{eq:truncated-entropy}
\frac{\dd}{\dd t}\KL(\rho_t^{\Lambda}\mid\pi)
=-\sum_{n\ge1}
\lambda_n q_r(\lambda_n)\chi(\lambda_n/\Lambda)
\abs{\ip{\nu_t^{\Lambda}}{\varphi_n}_{L^2(\pi)}}^2,
\end{equation}
and consequently
\begin{equation}
\label{eq:truncated-population-rate}
\D_{r,\pi}^{\Lambda}(\rho_t^{\Lambda},\pi)^2
\le \frac{\KL(\rho_0\mid\pi)}{\lambda_1t},
\qquad t>0.
\end{equation}
All constants are independent of $\Lambda$.
\end{proposition}

\begin{proof}
Global particle existence follows from the uniform bounded Lipschitz interaction, and global smooth population existence from \Cref{prop:compact-pop-wp}.  Finite-sum differentiation gives \eqref{eq:truncated-target-energy} as in \Cref{prop:common-energy}.  For entropy, the regularization argument in \Cref{app:entropy-chain-rule} applies because the characteristic formula preserves nonnegativity and the solution is smooth and bounded on every finite interval, including when $f_0$ vanishes.  Weighted integration by parts gives \eqref{eq:truncated-entropy}.  Since
$\lambda_nq_r(\lambda_n)\chi(\lambda_n/\Lambda)
\ge\lambda_1q_r(\lambda_n)\chi(\lambda_n/\Lambda)$,
integration and monotonicity give \eqref{eq:truncated-population-rate}.
The constants in these dissipation and relaxation estimates are independent of $\Lambda$.  Nonnegativity comes from the characteristic formula, so it does not require the cutoff multiplier to preserve positivity.
\end{proof}

\subsection{Uniform finite-time mean-field approximation}

Introduce the interaction field
\begin{equation}
\label{eq:b-Lambda}
b_{\Lambda}(x,y)=-\nabla_x\G_{r,\pi}^{\Lambda}(x,y).
\end{equation}
The large smoothing threshold in \Cref{prop:compact-regularity} implies
\begin{equation}
\label{eq:b-uniform-C1}
\sup_{\Lambda\ge\lambda_1}
\norm{b_{\Lambda}}_{C^1(\X\times\X)}<\infty.
\end{equation}
The same spectral estimates show that the feature map
\begin{equation}
\label{eq:feature-map-Lambda}
\Psi_{\Lambda}(x)
=\bigl(
q_r(\lambda_n)^{1/2}
\chi(\lambda_n/\Lambda)^{1/2}
\varphi_n(x)
\bigr)_{n\ge1}
\in\ell^2
\end{equation}
is uniformly bounded and Lipschitz.  Moreover,
\begin{equation}
\label{eq:D-feature-Lambda}
\D_{r,\pi}^{\Lambda}(\mu,\nu)
=\left\|
\int_{\X}\Psi_{\Lambda}\,\dd(\mu-\nu)
\right\|_{\ell^2}.
\end{equation}

These are precisely the interaction and feature bounds used in the full-kernel coupling proof.  Couple the particles to independent population characteristics $\bar X_i^{\Lambda}$ with common law $\rho_t^{\Lambda}$ and the same initial positions.  The proof in \Cref{app:finite-mode-mean-field} then gives the following estimate with constants independent of the cutoff.

\begin{proposition}[Finite-time truncated propagation of chaos]
\label{prop:finite-mode-finite-time}
There are constants $C_0,C_1<\infty$, depending on
$(\X,\pi,r,f_0)$ but independent of $N$, $T$, and
$\Lambda\ge\lambda_1$, such that
\begin{equation}
\label{eq:finite-mode-finite-time}
\E\sup_{0\le t\le T}
\D_{r,\pi}^{\Lambda}
(\mu_t^{N,\Lambda},\rho_t^{\Lambda})^2
\le \frac{C_0}{N}e^{C_1T}.
\end{equation}
\end{proposition}

As for the full kernel, this estimate will be used only up to a time of order $\log N$; target-energy monotonicity controls the later times.

\subsection{Stitching in the retained feature geometry}

Although the truncated discrepancy is only a seminorm on signed measures, it retains the triangle inequality.  Together with target-energy monotonicity, this is all that \Cref{lem:stitching} uses.  Therefore, for every $t\ge T$,
\begin{equation}
\label{eq:finite-mode-stitching}
\D_{r,\pi}^{\Lambda}
(\mu_t^{N,\Lambda},\rho_t^{\Lambda})
\le
\D_{r,\pi}^{\Lambda}
(\mu_T^{N,\Lambda},\rho_T^{\Lambda})
+2\D_{r,\pi}^{\Lambda}(\rho_T^{\Lambda},\pi).
\end{equation}
Combining \eqref{eq:finite-mode-finite-time},
\eqref{eq:truncated-population-rate}, and
\eqref{eq:finite-mode-stitching} gives
\begin{equation}
\label{eq:finite-mode-preopt}
\E\sup_{t\ge0}
\D_{r,\pi}^{\Lambda}
(\mu_t^{N,\Lambda},\rho_t^{\Lambda})
\le
\frac{C}{\sqrt N}e^{CT}+CT^{-1/2}.
\end{equation}
The same choice $T=c\log(2+N)$ as in \Cref{sec:relative-stability}, with $c>0$ small and independent of $\Lambda$, proves \eqref{eq:finite-mode-alltime-D}.

\subsection{Recovering unresolved spatial scales}

It remains to recover Lipschitz observables from the retained features.  A spectral cutoff at eigenvalue $\Lambda$ resolves spatial scales of order $\Lambda^{-1/2}$.  The next estimate separates the discrepancy in the retained modes from the error below that scale.

\begin{proposition}[Finite-resolution interpolation]
\label{prop:finite-resolution}
There is $C<\infty$, independent of $\Lambda\ge\lambda_1$, such that for all
$\mu,\nu\in\Pcal(\X)$,
\begin{equation}
\label{eq:finite-resolution}
\Wone(\mu,\nu)
\le C\left[
\Lambda^{(r+1)/2}
\D_{r,\pi}^{\Lambda}(\mu,\nu)
+\Lambda^{-1/2}
\right].
\end{equation}
\end{proposition}

\begin{proof}
Choose an auxiliary smoother supported entirely among the fully retained modes: let $\theta\in C_c^\infty([0,1))$ equal one near zero, with support contained in the region on which $\chi\equiv1$, and set
$P_{\le\Lambda}=\theta(\A_\pi/\Lambda)$.  For a
$1$-Lipschitz function $\phi$ normalized by $\pi(\phi)=0$,
\begin{equation}
\label{eq:finite-resolution-split}
(\mu-\nu)(\phi)
=(\mu-\nu)(P_{\le\Lambda}\phi)
+(\mu-\nu)((\Id-P_{\le\Lambda})\phi).
\end{equation}
The localized multiplier kernel and preservation of constants in \citet[Theorem~3.4]{coulhon2012heat} give, by integration against $\phi(y)-\phi(x)$,
\begin{equation}
\label{eq:spectral-lip-approx}
\norm{(\Id-P_{\le\Lambda})\phi}_{L^\infty}
\le C\Lambda^{-1/2}\Lip(\phi).
\end{equation}
For the retained term, spectral Cauchy--Schwarz and
$q_r(\lambda)^{-1}\le C\Lambda^{r+1}$ on
$\supp\theta(\lambda/\Lambda)$ give
\begin{equation}
\label{eq:finite-resolution-low}
\abs{(\mu-\nu)(P_{\le\Lambda}\phi)}
\le C\Lambda^{(r+1)/2}
\D_{r,\pi}^{\Lambda}(\mu,\nu)
\norm{\phi}_{L^2(\pi)}.
\end{equation}
Compactness and the normalization of $\phi$ bound the final norm uniformly.  Taking the supremum over $1$-Lipschitz functions proves the claim.  The multiplier estimate is recalled in \Cref{app:finite-resolution-proof}.
\end{proof}

Apply \Cref{prop:finite-resolution} with
$\mu=\mu_{t_N}^{N,\Lambda_N}$ and $\nu=\pi$, then use
\begin{equation}
\label{eq:finite-mode-target-triangle}
\D_{r,\pi}^{\Lambda_N}
(\mu_{t_N}^{N,\Lambda_N},\pi)
\le
\D_{r,\pi}^{\Lambda_N}
(\mu_{t_N}^{N,\Lambda_N},\rho_{t_N}^{\Lambda_N})
+
\D_{r,\pi}^{\Lambda_N}
(\rho_{t_N}^{\Lambda_N},\pi).
\end{equation}
Equations \eqref{eq:finite-mode-alltime-D} and
\eqref{eq:truncated-population-rate} give
\eqref{eq:finite-mode-last-bound}.  Balancing its two terms gives
\eqref{eq:optimal-Lambda}--\eqref{eq:finite-mode-optimized}, and Weyl counting gives \eqref{eq:mode-count}.  This proves
\Cref{thm:finite-mode-main}.

\subsection{Feature-factorized implementation}

Let $I_\Lambda=\{n\ge1:\lambda_n\le2\Lambda\}$ and
$L_\Lambda=\abs{I_\Lambda}$.  Define the empirical feature errors
\begin{equation}
\label{eq:empirical-features-Lambda}
m_n^{N,\Lambda}(t)
=\frac1N\sum_{j=1}^N\varphi_n(X_j^{N,\Lambda}(t)).
\end{equation}
Then
\begin{equation}
\label{eq:factorized-truncated-update}
\dot X_i^{N,\Lambda}
=-\sum_{n\in I_\Lambda}
q_r(\lambda_n)\chi(\lambda_n/\Lambda)
 m_n^{N,\Lambda}
\nabla\varphi_n(X_i^{N,\Lambda}).
\end{equation}
Once the feature values and gradients have been evaluated, all coefficients
$m_n^{N,\Lambda}$ are formed in $O(NL_\Lambda)$ operations and the velocity sum costs $O(NL_\Lambda d)$.  No all-pairs $N^2$ interaction and no dense $d\times d$ matrix kernel need be formed.

\subsection{Approximate eigenpairs and structural residuals}
\label{sec:approx-eigenpairs}

The preceding implementation formula uses exact target eigenpairs.  If numerical eigenpairs satisfy
\begin{equation}
\label{eq:approx-eigenpair}
\A_\pi\widetilde\varphi_n
=\widetilde\lambda_n\widetilde\varphi_n+r_n,
\end{equation}
then the corresponding matrix lift satisfies
\begin{equation}
\label{eq:approx-intertwining-residual}
\operatorname{div}_{\pi,y}\widetilde\K(x,y)
=-\nabla_x\widetilde\G(x,y)
-\sum_n
\frac{q_r(\widetilde\lambda_n)}{\widetilde\lambda_n}
\nabla\widetilde\varphi_n(x)r_n(y).
\end{equation}
The residual perturbs the velocity and the exact Lyapunov identity at every time.  Extending the theorem to numerical eigenpairs therefore requires bounds on $r_n$ in a norm that controls evaluation against empirical measures, uniformly along the dynamics; see \Cref{sec:discussion}.

\section{Confining targets on \texorpdfstring{$\R^d$}{Euclidean space}}
\label{sec:euclidean-targets}

On an unbounded state space, the common energy must also prevent mass from escaping to infinity. Quadratic confinement makes this possible: the Green--Bessel discrepancy controls a second moment, which gives compactness in $\Wone$ and global continuation of every finite particle system. We first prove these static conclusions of \Cref{thm:euclidean-main}. We then use the population and transport hypotheses in \Cref{ass:euclidean-transport} to obtain finite-time comparison. Those hypotheses are the additional input for the conditional dynamical extension; confinement alone is not asserted to verify them.

\subsection{Target class and spectral gap}

\begin{assumption}[Smooth quadratic-type confinement]
\label{ass:euclidean-target}
Let $\X=\R^d$, fix $0<\kappa\le K<\infty$, and let $U,W\in C^\infty(\R^d)$ satisfy
\begin{equation}
\label{eq:euclidean-V}
V=U+W,\qquad \kappa\Id\le D^2U\le K\Id,
\end{equation}
where $W$ is bounded and
\begin{equation}
\label{eq:euclidean-finite-regularity}
\norm{D^jU}_\infty<\infty\ (j\ge3),\qquad
\norm{D^jW}_\infty<\infty\ (j\ge1).
\end{equation}
Fix an integer smoothing parameter
\begin{equation}
\label{eq:euclidean-r-threshold}
r>2d+10.
\end{equation}
\end{assumption}

The derivative bounds allow cutoff approximation in the operator domains used below. The smoothing threshold leaves enough Sobolev regularity for point and derivative evaluation, and also makes the kernel trace finite. This sufficient threshold is not asserted to be optimal.

Let $\nu(\dd x)=Z_U^{-1}e^{-U(x)}\dd x$. The curvature criterion gives a Poincar\'e gap at least $\kappa$ for $\nu$ \citep[Section~4.8]{bakry2014analysis}. The following is the bounded-perturbation comparison associated with the Holley--Stroock principle \citep{holley1987logarithmic}.

\begin{proposition}[Bounded-perturbation gap]
\label{prop:euclidean-gap}
Under \eqref{eq:euclidean-V},
\begin{equation}
\label{eq:HS-gap}
\lambda_1(\pi)\ge\kappa e^{-\osc W}.
\end{equation}
\end{proposition}
\begin{proof}
Put $a=\dd\pi/\dd\nu$. For every smooth $g$,
\[
\operatorname{Var}_\pi(g)
\le(\sup a)\operatorname{Var}_\nu(g)
\le\frac{\sup a}{\kappa\inf a}
\int|\nabla g|^2\,\dd\pi.
\]
Here $(\sup a)/(\inf a)=e^{\osc W}$.
\end{proof}

Boundedness of $W$ supplies the explicit gap comparison above. The kernel and moment arguments need only bounded derivatives of $W$ of positive order, together with a separately established positive gap. Under this broader condition, $\pi$ still has Gaussian upper tails, $\nabla V$ grows linearly, and
\begin{equation}
\label{eq:euclidean-confinement-bounds}
V(x)\ge c|x|^2-C,\qquad
|\nabla V(x)|^2\ge c|x|^2-C,\qquad
\norm{D^2V}_\infty<\infty.
\end{equation}
This observation will be used for Gaussian mixtures.

\subsection{Local kernels and a dual discrepancy}

To use the same discrepancy for densities and atoms, we regard a measure as acting on an energy space of test functions. Write
\[
\mathcal H=\operatorname{Dom}(\Q_{r,\pi}^{-1/2}),\qquad
\|h\|_{\mathcal H}=\|\Q_{r,\pi}^{-1/2}h\|_2,
\]
on centered functions. The spectral gap makes this norm equivalent to the centered $\operatorname{Dom}((\Id+\A_\pi)^{(r+1)/2})$ norm. In the Euclidean setting the discrepancy on arbitrary probability measures is the extended dual norm
\begin{equation}
\label{eq:euclidean-dual-domain}
\D_{r,\pi}(\mu,\nu)
=\sup_{\phi\in C_c^\infty,\ \|\phi-\pi(\phi)\|_{\mathcal H}\le1}
|\mu(\phi)-\nu(\phi)|.
\end{equation}
Its value may be infinite. Compactly supported tests make the definition meaningful for every probability measure; local point evaluation and the test-function core will recover the kernel and spectral formulas for empirical measures and for measures with $L^2(\pi)$ density.

For $L\ge0$, let $N_{\A}(L)=\#\{n\ge0:\lambda_n\le L\}$ count eigenvalues with multiplicity, including the constant mode $\lambda_0=0$.

\begin{proposition}[Euclidean kernel and test-function core]
\label{prop:euclidean-kernel}
Under the derivative bounds \eqref{eq:euclidean-finite-regularity}, the smoothing range \eqref{eq:euclidean-r-threshold}, the confinement bounds \eqref{eq:euclidean-confinement-bounds}, and a positive gap, the centered functions $\phi-\pi(\phi)$, $\phi\in C_c^\infty$, are dense in $\mathcal H$. The scalar kernel is positive definite, symmetric, and locally $C^2$ in each variable; its mixed derivatives of order at most two in each variable have locally uniformly convergent spectral series. The matrix lift is well defined and satisfies \eqref{eq:K-intertwine-framework}. Moreover,
\begin{equation}
\label{eq:euclidean-counting}
N_{\A}(L)\le C(1+L)^d,
\end{equation}
and
\begin{equation}
\label{eq:euclidean-initial-diagonal-bound}
\int\G_{r,\pi}(x,x)\,\rho_0(\dd x)
\le\|f_0\|_\infty\operatorname{Tr}\Q_{r,\pi}<\infty
\end{equation}
for every bounded density ratio $\rho_0=f_0\pi$.
\end{proposition}

The proof in \Cref{app:euclidean-calculus} first compares the conjugated generator with a harmonic oscillator to count eigenvalues, then uses interior elliptic estimates to construct the local evaluation functionals. Local regularity gives local existence for the particle ordinary differential equation; the moment estimate below will give continuation. Working with the dual norm also avoids requiring a global polynomial bound for the kernel, which ground-state conjugation need not preserve.

\subsection{Moment coercivity and separation}

The key observation is that the centered quadratic function belongs to $\mathcal H$. To turn this fact into a moment bound for an unknown measure, we first approximate it by bounded tests; this avoids assuming the integrability we want to prove.

\begin{proposition}[Coercive moment duality]
\label{prop:euclidean-moment}
For $p(x)=|x|^2$, one has $p-\pi(p)\in\mathcal H$. Every probability measure of finite dual energy satisfies
\begin{equation}
\label{eq:moment-duality}
|\mu(p)-\pi(p)|
\le \D_{r,\pi}(\mu,\pi)\,
\|p-\pi(p)\|_{\mathcal H}.
\end{equation}
In particular, finite target energy implies a finite second moment. Every target-energy sublevel is compact in $\Wone$, and every empirical solution is global.
\end{proposition}

\begin{proof}
Choose smooth bounded truncations $p_R\uparrow p$ which are constant outside a ball and satisfy
\[
p_R-\pi(p_R)\longrightarrow p-\pi(p)
\quad\hbox{in }\mathcal H.
\]
Such truncations, including their graph-norm convergence, are constructed in \Cref{app:euclidean-calculus}. Each centered $p_R$ is a centered compactly supported test function after subtracting its constant value at infinity. Therefore \eqref{eq:euclidean-dual-domain} gives the dual inequality for $p_R$. The bounded form norms and monotone convergence first establish $\mu(p)<\infty$ and then give \eqref{eq:moment-duality}.

This moment estimate has two consequences. First, a uniform second-moment bound gives tightness and uniform integrability of first moments. The dual norm is lower semicontinuous for weak convergence, since it is a supremum of continuous functionals of bounded continuous tests; hence the sublevel is closed and compact in $\Wone$.

Second, for an empirical solution, the finite-particle Hilbert-space chain rule gives the target-energy identity \eqref{eq:common-energy} on its maximal existence interval. Thus $N^{-1}\sum_i|X_i(t)|^2$ stays bounded there. For fixed $N$ this bounds every particle, so local Lipschitzness of the force gives continuation for all times.
\end{proof}

\begin{proposition}[Separation and topological compatibility]
\label{prop:euclidean-separation}
The dual discrepancy separates probabilities of finite energy. If $\mu_n,\nu_n$ belong to a common target-energy sublevel and $\D_{r,\pi}(\mu_n,\nu_n)\to0$, then $\Wone(\mu_n,\nu_n)\to0$.
\end{proposition}
\begin{proof}
Zero discrepancy implies equality on every compactly supported smooth test, which determines a finite Borel measure. If the second assertion failed, compactness would give a subsequence converging in $\Wone$ to two distinct limits. Lower semicontinuity of \eqref{eq:euclidean-dual-domain} would give zero discrepancy between those limits, a contradiction.
\end{proof}

\subsection{From transport control to finite-time comparison}

The static estimates now supply the topology and global particle flow required by the abstract theorem. To compare that flow with the population, it remains to control the transport term in the relative-energy identity. The following condition has two jobs: bounded population norms justify the energy and entropy chain rules, while the operator bounds control how transport changes the error's energy norm. The space $\mathcal H_+$ supplies the extra derivative needed to apply transport; the relative potentials of empirical errors belong to this space, as proved in \Cref{app:euclidean-calculus}.

\begin{assumption}[Euclidean population and transport input]
\label{ass:euclidean-transport}
For the specified initial density $f_0$, assume the population equation has a unique global nonnegative classical solution $\rho_t=f_t\pi$. On every finite interval, $f_t$ and $v_t=-\nabla\Q_{r,\pi}(f_t-1)$ have bounded $C_b^1(\R^d)$ norms. On the centered energy space $\mathcal H$ defined above, let $S=\Q_{r,\pi}^{-1/2}$ and set
\[
\mathcal H_+
=\operatorname{Dom}((\Id+\A_\pi)^{(r+2)/2})\cap L^2_0(\pi).
\]
Define the centered transport operator by
\begin{equation}
\label{eq:euclidean-centered-transport}
B_t h=v_t\cdot\nabla h-\pi(v_t\cdot\nabla h).
\end{equation}
Assume $B_t:\mathcal H_+\to\mathcal H$ is continuous, with operator norm bounded on every finite interval. The operator
\[
[S,B_t]S^{-1}g
=S B_t S^{-1}g-B_tg,
\qquad g\in\operatorname{Dom}(\A_\pi^{1/2})\cap L^2_0(\pi),
\]
is well defined on this domain and extends to a bounded operator on $L^2_0(\pi)$, with norm bounded on every finite interval.
\end{assumption}

The mapping property makes $B_t h$ an admissible energy-space test. The commutator bound then controls its pairing with the relative potential, as the next proposition shows.

\begin{proposition}[Sufficient Euclidean transport condition]
\label{prop:euclidean-transport}
Suppose the population and operator conditions in \Cref{ass:euclidean-transport} hold. For $t\le T$ and $h\in\mathcal H_+$,
\begin{equation}
\label{eq:euclidean-transport-form}
|\langle h,B_t h\rangle_{\mathcal H}|
\le c_T\|h\|_{\mathcal H}^2,
\qquad
c_T=\tfrac12\sup_{t\le T}\|\divpi v_t\|_\infty
+\sup_{t\le T}\|[S,B_t]S^{-1}\|_{2\to2}.
\end{equation}
For every empirical $\mu$, the relative potential
$h=\PhiPot_{\mu-\rho_t}$ belongs to $\mathcal H_+$, and
\begin{equation}
\label{eq:euclidean-actual-pairing}
(\mu-\rho_t)(v_t\cdot\nabla h)
=\langle h,B_t h\rangle_{\mathcal H}.
\end{equation}
Consequently the finite-horizon estimate \eqref{eq:abstract-form} holds.
\end{proposition}

\begin{proof}
Conjugating transport by $S$ separates its symmetric part, controlled by the weighted divergence, from the commutator. The centered projection does not change an inner product against $Sh$. Integration by parts and the boundedness of $v_t$ and $\divpi v_t$ give
\[
\langle Sh,B_t Sh\rangle_2
=-\tfrac12\int(\divpi v_t)|Sh|^2\,\dd\pi.
\]
The definition of the commutator therefore gives
\[
\langle h,B_t h\rangle_{\mathcal H}
=\langle Sh,B_t Sh\rangle_2
+\langle Sh,[S,B_t]S^{-1}Sh\rangle_2,
\]
which proves \eqref{eq:euclidean-transport-form}. The domain and Riesz-pairing statements, proved in \Cref{app:euclidean-calculus}, justify this calculation for the actual empirical errors. Dropping the nonpositive dissipation term in \eqref{eq:relative-energy} and applying \eqref{eq:euclidean-transport-form} gives Gr\"onwall's inequality with coefficient $2c_T$. Since $c_T\ge0$,
\[
\sup_{0\le t\le T}\D_{r,\pi}(\mu_t^N,\rho_t)^2
\le e^{2c_TT}\D_{r,\pi}(\mu_0^N,\rho_0)^2.
\]
Taking expectations under the independent initialization gives \eqref{eq:abstract-form}: the initial expectation is $O(N^{-1})$ by \eqref{eq:euclidean-initial-diagonal-bound} and \Cref{prop:initial-fluctuation}.
\end{proof}

The weighted divergence is explicitly
\begin{equation}
\label{eq:euclidean-v-and-div}
v_t=-\nabla\Q_{r,\pi}(f_t-1),\qquad
\divpi v_t=(\Id+\A_\pi)^{-r}(f_t-1).
\end{equation}
The resolvent is an $L^\infty$ contraction, so the divergence contribution in \eqref{eq:euclidean-transport-form} is already controlled by the assumed finite-time bound on $f_t$. The operator assumptions supply the other contribution and justify its pairing with empirical errors.

\subsection{The conditional Euclidean theorem}

\begin{theorem}[Conditional extension to confining Euclidean targets]
\label{thm:euclidean-main}
Under \Cref{ass:euclidean-target}, the scalar and matrix kernels have the local regularity and finite diagonal specified in \Cref{prop:euclidean-kernel}; the dual discrepancy controls a second moment, separates measures, and has compact target-energy sublevels in $\Wone$. The empirical particle system is globally well posed for every finite configuration, and the gap satisfies \eqref{eq:HS-gap}.
Let $\rho_0=f_0\pi$, with $f_0\in C_b^\infty$, $f_0\ge0$, $\pi(f_0)=1$, and initialize the particles independently from $\rho_0$. If \Cref{ass:euclidean-transport} also holds for this $f_0$, all conclusions of \Cref{thm:abstract-main} hold. In particular,
\begin{equation}
\label{eq:euclidean-pop-rate}
\D_{r,\pi}(\rho_t,\pi)^2
\le\frac{e^{\osc W}}{\kappa t}\KL(\rho_0\mid\pi),\qquad t>0.
\end{equation}
\end{theorem}

\begin{proof}[Proof of \Cref{thm:euclidean-main}]
We assemble the inputs to \Cref{thm:abstract-main}. \Cref{prop:euclidean-gap,prop:euclidean-kernel} give the gap, local kernel regularity, the test-function core, and the finite initial diagonal integral. \Cref{prop:euclidean-moment,prop:euclidean-separation} give moment coercivity, particle nonexplosion, separation, and topological compatibility.

For the dynamical inputs, \Cref{ass:euclidean-transport} supplies the global classical population flow. Its bounded $C^1$ norms on finite intervals justify the energy and entropy identities by the Hilbert-space chain rule and the cutoff argument in \Cref{app:euclidean-calculus}; \Cref{prop:euclidean-transport} gives finite-time comparison. The abstract theorem now extends this comparison to all times by population relaxation and monotonicity of the common energy. Inserting \eqref{eq:HS-gap} into the population bound gives \eqref{eq:euclidean-pop-rate}.
\end{proof}

\section{Common-covariance Gaussian mixtures}
\label{sec:gaussian-mixtures}

A common-covariance Gaussian mixture shows how the Euclidean construction accommodates several wells. Its potential has a quadratic part plus a perturbation with bounded derivatives of positive order. Thus the kernel and moment arguments of \Cref{sec:euclidean-targets} apply once we supply a positive spectral gap. The task here is to prove an explicit gap bound by comparing variation within components with variation between components.

Let
\begin{equation}
\label{eq:mixture-target}
\pi=\sum_{j=1}^Jw_j\gamma_j,
\qquad
\gamma_j=\mathcal N(m_j,B^{-1}),
\qquad
B=B^\top>0,
\end{equation}
where $w_j>0$ and $\sum_jw_j=1$.  Set
\begin{equation}
\label{eq:mixture-parameters}
b=\lambda_{\min}(B),
\qquad
w_\ast=\min_jw_j,
\qquad
D_\ast^2=
\max_{i,j}(m_i-m_j)^\top B(m_i-m_j).
\end{equation}
Up to an additive constant, the target potential is
\begin{equation}
\label{eq:mixture-potential}
V(x)=\frac12x^\top Bx
-\log\left(
\sum_{j=1}^J
w_j\exp\left(x^\top Bm_j-\frac12m_j^\top Bm_j\right)
\right).
\end{equation}
The second term is generally unbounded along rays.  Its gradient, however, is a convex combination of the finitely many vectors $-Bm_j$, and every derivative of order at least two is a bounded cumulant of those vectors.  Thus the target belongs to the class
\begin{equation}
\label{eq:quadratic-Lipschitz-class}
V=U+W,
\qquad
U(x)=\frac12x^\top Bx,
\qquad
\nabla W\in L^\infty,
\quad
D^kW\in L^\infty\ (k\ge2),
\end{equation}
covered by the extension following \Cref{prop:euclidean-gap}. Boundedness of $W$ itself fails, so the mixture requires the separate gap estimate below.

\subsection{An explicit mixture Poincar\'e inequality}

Poincar\'e and logarithmic Sobolev inequalities for mixtures have been studied in detail; see \citet{schlichting2019poincare}, including the equal-covariance Gaussian examples. The elementary bound below supplies an explicit constant for this application; its dependence on component separation is not claimed to be optimal.

\begin{proposition}[Poincar\'e gap for a common-covariance mixture]
\label{prop:mixture-gap}
The mixture \eqref{eq:mixture-target} satisfies
\begin{equation}
\label{eq:mixture-gap}
\lambda_1(\pi)
\ge
\frac{b}{
1+\dfrac{D_\ast^2}{2w_\ast}e^{D_\ast^2/8}
}.
\end{equation}
\end{proposition}

\begin{proof}
The variance decomposes into a part controlled by each Gaussian Poincar\'e inequality and a part measuring differences of component means. We control the latter by moving one Gaussian mean to another along a line segment.

For a smooth $f$, let
\begin{equation}
\label{eq:component-means}
a_j=\int f\,\dd\gamma_j,
\qquad
\bar a=\sum_jw_ja_j.
\end{equation}
The law of total variance gives
\begin{equation}
\label{eq:mixture-total-variance}
\operatorname{Var}_\pi(f)
=\sum_jw_j\operatorname{Var}_{\gamma_j}(f)
+\frac12\sum_{i,j}w_iw_j(a_i-a_j)^2.
\end{equation}
Each component has Poincar\'e gap at least $b$ by the Gaussian curvature criterion \citep[Section~4.8]{bakry2014analysis}, so
\begin{equation}
\label{eq:component-gap}
\sum_jw_j\operatorname{Var}_{\gamma_j}(f)
\le\frac1b\int\abs{\nabla f}^2\,\dd\pi.
\end{equation}

To control the second term in \eqref{eq:mixture-total-variance}, fix $i,j$, set $\delta=m_j-m_i$, and interpolate by
\begin{equation}
\label{eq:mixture-interpolation}
\gamma_t=\mathcal N(m_i+t\delta,B^{-1}),
\qquad 0\le t\le1.
\end{equation}
Differentiation under the Gaussian expectation gives
\begin{equation}
\label{eq:mean-interpolation-derivative}
a_j-a_i
=\int_0^1\int\delta\cdot\nabla f\,\dd\gamma_t\,\dd t.
\end{equation}
Since $\abs\delta^2\le b^{-1}\delta^\top B\delta$, Cauchy--Schwarz yields
\begin{equation}
\label{eq:mean-difference-gradient}
(a_i-a_j)^2
\le\frac{D_{ij}^2}{b}
\int_0^1\int\abs{\nabla f}^2\,\dd\gamma_t\,\dd t,
\qquad
D_{ij}^2=\delta^\top B\delta.
\end{equation}
The remaining step is to compare the interpolating measure with $\pi$. Common covariance gives the density identity
\begin{equation}
\label{eq:gaussian-geometric-interpolation}
\gamma_t
=e^{\frac12t(1-t)D_{ij}^2}
\gamma_i^{1-t}\gamma_j^t.
\end{equation}
Weighted arithmetic--geometric mean and
$\gamma_k\le w_\ast^{-1}\pi$ imply
\begin{equation}
\label{eq:gamma-t-comparison}
\gamma_t
\le e^{D_{ij}^2/8}
\bigl[(1-t)\gamma_i+t\gamma_j\bigr]
\le\frac{e^{D_{ij}^2/8}}{w_\ast}\pi.
\end{equation}
Combining \eqref{eq:mean-difference-gradient} and
\eqref{eq:gamma-t-comparison},
\begin{equation}
\label{eq:between-component-bound}
(a_i-a_j)^2
\le
\frac{D_\ast^2e^{D_\ast^2/8}}{b w_\ast}
\int\abs{\nabla f}^2\,\dd\pi.
\end{equation}
Insert \eqref{eq:component-gap} and
\eqref{eq:between-component-bound} into
\eqref{eq:mixture-total-variance} and use
$\sum_{i,j}w_iw_j=1$.  This gives
\begin{equation}
\operatorname{Var}_\pi(f)
\le\frac1b\left(
1+\frac{D_\ast^2}{2w_\ast}e^{D_\ast^2/8}
\right)
\int\abs{\nabla f}^2\,\dd\pi,
\end{equation}
which is equivalent to \eqref{eq:mixture-gap}.
\end{proof}

\subsection{Static guarantees and conditional dynamics}

The gap bound completes the static hypotheses: \eqref{eq:mixture-potential} has quadratic confinement, $\nabla V(x)=Bx+O(1)$, and all positive-order derivatives of $W$ are bounded. Consequently the local kernels, coercive moment bound, compactness and separation, and global existence of every finite particle system follow. The all-time comparison with the population additionally requires \Cref{ass:euclidean-transport} for the mixture and initial density under consideration; that condition is not proved here.

\begin{corollary}[Gaussian mixtures]
\label{cor:mixture-main-summary}
Let $\pi$ be the mixture \eqref{eq:mixture-target}, with parameters \eqref{eq:mixture-parameters}.
For every integer $r>2d+10$, the static kernel, moment, separation, and empirical nonexplosion assertions of \Cref{thm:euclidean-main} hold, with the gap bound \eqref{eq:mixture-gap}.
For an initial density as in \Cref{thm:euclidean-main}, the conclusions of \Cref{thm:abstract-main} hold if the population and transport input in \Cref{ass:euclidean-transport} is verified for this mixture and this initial density. Under that additional hypothesis,
\begin{equation}
\label{eq:mixture-population-rate}
\D_{r,\pi}(\rho_t,\pi)^2
\le\frac1{bt}\left(1+\frac{D_*^2}{2w_*}e^{D_*^2/8}\right)
\KL(\rho_0\mid\pi).
\end{equation}
\end{corollary}

\begin{proof}[Proof of \Cref{cor:mixture-main-summary}]
Apply \Cref{prop:mixture-gap} and the static proofs in \Cref{sec:euclidean-targets,app:euclidean-calculus} to \eqref{eq:mixture-potential}. Under \Cref{ass:euclidean-transport}, the proof of \Cref{thm:euclidean-main} then applies with the mixture gap in place of the bounded-perturbation gap. Substituting \eqref{eq:mixture-gap} in the abstract population bound gives the stated rate.
\end{proof}

If the initial law is one mixture component, say
$\rho_0=\gamma_k$, then
\begin{equation}
\label{eq:component-ratio}
\frac{\dd\rho_0}{\dd\pi}
=\frac{\gamma_k}{\sum_jw_j\gamma_j}
\le\frac1{w_k},
\end{equation}
so the bounded density-ratio hypothesis is automatic.

\begin{remark}[Metastability is visible in the constant]
The lower bound \eqref{eq:mixture-gap} deteriorates exponentially when components separate in the Mahalanobis distance.  This is consistent with the slow interwell relaxation of a highly separated mixture.  When the additional population and transport hypotheses hold, the all-time particle--population conclusion applies; the population time required to reach a prescribed target accuracy may nevertheless be large.
\end{remark}

\section{Scope and further directions}
\label{sec:discussion}

The compact theorem controls the empirical measure at every physical time, with the time supremum inside the expectation.  Its last-iterate consequence is a joint limit in particle number and time: for fixed $N$, the measure remains $N$-atomic.  The proof uses finite-time approximation and dissipation of a common target energy, without a contraction estimate for the interacting $N$-particle flow.  This places the next questions at the points where that structure may change.

\subsection{Spectral approximation and time discretization}

The exact-mode theorem preserves the Stein identity and both dissipation laws while introducing a controlled resolution error.  An eigensolver or learned spectral surrogate introduces a different error.  As \eqref{eq:approx-intertwining-residual} shows, a generator residual changes the force itself and hence the target-energy law.  To extend the all-time theorem, one would need to bound that residual on empirical measures and control its accumulated energy defect.  Such a result would add a certified spectral error to the particle and resolution errors in \eqref{eq:finite-mode-last-bound}.

A time-discrete scheme faces the same requirement.  An integrator that preserves target-energy monotonicity, or satisfies a one-step inequality with a summable defect, could retain the comparison-time argument.  Since the finite-mode energy is an explicit quadratic function of the retained moments, discrete-gradient or suitably implicit schemes are natural candidates.  Their stability and consistency would require a separate analysis.

\subsection{Unbounded targets and weaker initial regularity}

For the confining Euclidean targets considered here, the static kernel and moment results provide well-defined empirical dynamics.  The remaining dynamical input is the global population regularity and transport condition in \Cref{ass:euclidean-transport}.  Proving this condition for concrete target families, including Gaussian mixtures, would turn the conditional extension into a dynamical theorem for those families.  The local kernel and duality arguments isolate this task from moment coercivity and particle nonexplosion.

The all-time principle also separates population attraction from its particular rate.  A weak Poincar\'e inequality or spectral profile could replace the gap in \eqref{eq:entropy-controls-D} if it yielded a population relaxation estimate in the same dissipated discrepancy.  More general confinement would require compatible kernel and moment estimates.  Extending initialization from bounded smooth density ratios to finite entropy would additionally require a regularity argument for the population flow.  These extensions preserve the organizing question: which target energy can be dissipated by both the population and the empirical dynamics?

\appendix
\section{Energy identities for measures and the entropy chain rule}
\label{app:atomic-regularization}

Kernel means place empirical measures and population laws in the same Hilbert space.  Once their time derivatives are identified there, the usual chain rule proves both the target-energy identity and the relative-energy identity used for Euclidean transport.  The final subsection supplies the separate entropy argument at zeros of the density.

\subsection{Kernel means and differentiation}

Write $\mathcal H=\Dom(\Q_{r,\pi}^{-1/2})$ with its centered inverse-form norm, and let $\Psi(x)=\G_{r,\pi}(\cdot,x)\in\mathcal H$.  Point evaluation gives
\[
\ip{\Psi(x)}{h}_{\mathcal H}=h(x),\qquad
\norm{\Psi(x)}_{\mathcal H}^2=\G_{r,\pi}(x,x).
\]
In the orthonormal basis $(\sqrt{q_r(\lambda_n)}\varphi_n)_{n\ge1}$ of $\mathcal H$, the coordinates of $\Psi(x)$ are $(\sqrt{q_r(\lambda_n)}\varphi_n(x))_{n\ge1}$.  This identifies it isometrically with the $\ell^2$ feature map in \eqref{eq:compact-feature-map}.
Under the local differentiated-kernel bounds proved in the target-specific sections, $\Psi$ is locally $C^1$ as a Hilbert-space-valued map.  Indeed, difference quotients converge in $\mathcal H$ by the locally uniform convergence of the mixed second derivatives of the positive spectral kernel.  For any probability measure for which the following Bochner integral exists, set
\[
 m_\mu=\int\Psi(x)\,\mu(\dd x).
\]
It exists if $\int\sqrt{\G_{r,\pi}(x,x)}\,\mu(\dd x)<\infty$.  Centering gives $m_\pi=0$, and the dual definition of the discrepancy yields
\[
 \D_{r,\pi}(\mu,\nu)=\norm{m_\mu-m_\nu}_{\mathcal H}.
\]
These statements apply to every compact-state probability measure, to every empirical measure, and to $f\pi$ with bounded $f$ in the Euclidean class because $\Tr\Q_{r,\pi}<\infty$.

For a particle curve, the finite sum defining $m_{\mu_t^N}$ is $C^1$ in $\mathcal H$ as long as the positions remain in a bounded set.  Its derivative satisfies
\begin{equation}
\label{eq:finite-rank-diff}
 \ip{\dot m_{\mu_t^N}}{h}_{\mathcal H}
 =\frac1N\sum_{i=1}^N\nabla h(X_i^N(t))\cdot\dot X_i^N(t).
\end{equation}
This identity includes all self-interactions.

For a classical population curve $\rho_t=f_t\pi$, suppose $f_t$ and $v_t$ are bounded on every finite time interval.  These bounds make the continuity equation a bounded functional on $\mathcal H$: for $h\in\mathcal H$,
\[
 \left|\int\nabla h\cdot v_t f_t\,\dd\pi\right|
 \le \norm{f_t}_{\infty}\norm{v_t}_{\infty}\norm{\nabla h}_{L^2(\pi)}
 \le C\norm{f_t}_{\infty}\norm{v_t}_{\infty}\norm{h}_{\mathcal H}.
\]
The last inequality follows spectrally from
$\lambda\le\lambda(1+\lambda)^r$.  The weak continuity equation, first on smooth test functions and then by their density in $\mathcal H$, therefore identifies a locally integrable $\mathcal H$-valued derivative of $m_{\rho_t}$.  In particular, this kernel mean is locally absolutely continuous and
\begin{equation}
\label{eq:population-mean-derivative}
 \ip{\dot m_{\rho_t}}{h}_{\mathcal H}
 =\int\nabla h\cdot v_t\,\dd\rho_t
\end{equation}
for almost every $t$.  Compact-state solutions have these bounds by \Cref{sec:compact-verification}; in the Euclidean extension they are part of the explicitly stated population hypothesis.

\subsection{Target and relative energy}

For either curve above, write $g_t=m_{\mu_t}=\PhiPot_{\mu_t-\pi}$.  The Hilbert-space chain rule gives
\[
 \frac12\frac{\dd}{\dd t}\norm{g_t}_{\mathcal H}^2
 =\ip{\dot g_t}{g_t}_{\mathcal H}
 =\int\nabla g_t\cdot v_{\mu_t}\,\dd\mu_t
 =-\int\abs{\nabla g_t}^2\,\dd\mu_t.
\]
For the particle system this is a classical identity up to its maximal existence time; the Euclidean moment bound then prevents finite-time escape.  For the population it holds almost everywhere, hence in integrated form.  This proves \eqref{eq:common-energy} under the stated hypotheses.

For the particle--population difference, set
\begin{equation}
\label{eq:eta-h-def}
\eta_t^N=\mu_t^N-\rho_t,
\qquad h_t^N=\PhiPot_{\eta_t^N}=m_{\mu_t^N}-m_{\rho_t}.
\end{equation}
Subtracting the two continuity equations and using
$v_{\mu_t^N}=v_{\rho_t}-\nabla h_t^N$ gives
\begin{equation}
\label{eq:relative-continuity}
\partial_t\eta_t^N+\nabla\cdot(\eta_t^N v_{\rho_t})
-\nabla\cdot(\mu_t^N\nabla h_t^N)=0.
\end{equation}
The same Hilbert-space chain rule separates the relative dissipation from transport by the population velocity:
\begin{align}
 \frac12\frac{\dd}{\dd t}\D_{r,\pi}(\eta_t^N)^2
 &=\int\nabla h_t^N\cdot v_{\mu_t^N}\,\dd\mu_t^N
   -\int\nabla h_t^N\cdot v_{\rho_t}\,\dd\rho_t\nonumber\\
\label{eq:relative-energy}
 &=-\int\abs{\nabla h_t^N}^2\,\dd\mu_t^N
   +\eta_t^N(v_{\rho_t}\cdot\nabla h_t^N).
\end{align}
The last term is a measure integral: its atomic part is finite, and its population part is integrable by the preceding bound.  To estimate it in the Hilbert norm, the centered transport function must also belong to that space.  When
$B_t h_t^N=v_{\rho_t}\cdot\nabla h_t^N-\pi(v_{\rho_t}\cdot\nabla h_t^N)$ belongs to $\mathcal H$, the representer identity gives
\[
 \eta_t^N(v_{\rho_t}\cdot\nabla h_t^N)
 =\ip{h_t^N}{B_t h_t^N}_{\mathcal H}.
\]
The Euclidean transport assumption specifies the domain and bound for this expression.  On compact manifolds, the direct coupling proof in \Cref{app:finite-mode-mean-field} supplies the finite-time comparison.

\subsection{The entropy chain rule at zero density}
\label{app:entropy-chain-rule}

The derivative of $s\log s$ is singular at zero, although the final dissipation involves only $\nabla f_t$.  Regularization makes this cancellation valid for classical nonnegative solutions.  On a compact state space, define
\[
 \beta_\eps(s)=(s+\eps)\log(s+\eps)-\eps\log\eps,
 \qquad \eps>0.
\]
Multiplying \eqref{eq:ratio-equation} by $\beta_\eps'(f_t)$, integrating against $\pi$, and using the continuity equation gives
\[
 \frac{\dd}{\dd t}\int\beta_\eps(f_t)\,\dd\pi
 =-\int\frac{f_t}{f_t+\eps}\,
       \nabla f_t\cdot\nabla\Q_{r,\pi}(f_t-1)\,\dd\pi.
\]
At every zero of a differentiable nonnegative function, its gradient is zero.  Thus the integrand converges pointwise everywhere to
$-\nabla f_t\cdot\nabla\Q_{r,\pi}(f_t-1)$.
On a finite interval compact-state classical regularity bounds this integrand uniformly.  Dominated convergence in time and space, together with uniform convergence of $\beta_\eps$ on bounded intervals, yields the integrated entropy identity.  Weighted integration by parts then gives \eqref{eq:entropy-dissipation}.  This proof does not require a positive lower bound on $f_0$.

On $\R^d$, the assumed finite-time $C_b^1$ bounds on the population density and velocity make the displayed gradient product integrable in time and space.  Spatial cutoffs have vanishing boundary contributions by Gaussian tails, first for fixed $\eps>0$; one then sends $\eps$ to zero as above.  Details are recorded in \Cref{app:euclidean-calculus}.  These regularity bounds are explicit dynamical hypotheses of the Euclidean theorem, separate from the static confinement assumptions.

\section{Compact spectral and heat estimates}
\label{app:compact-analysis}

Two estimates support the compact theorem.  Summable spectral tails justify differentiation of the kernels and provide bounds uniform in the cutoff.  The short-time displacement of the Langevin diffusion controls the heat approximation of Lipschitz functions.  The finite-time coupling argument using these estimates follows in \Cref{app:finite-mode-mean-field}.

\subsection{Uniform derivative tails}

The differentiated local Weyl estimate \eqref{eq:local-weyl-derivative} implies, for any $a>d/2+k$ and sufficiently large $L$,
\begin{equation}
\label{eq:compact-derivative-tail}
\sup_x\sum_{\lambda_n>L}(1+\lambda_n)^{-a}
\abs{\nabla^k\varphi_n(x)}^2
\le C_{a,k}L^{d/2+k-a}.
\end{equation}
Indeed, split the sum into $2^jL<\lambda_n\le2^{j+1}L$, bound each shell by
$C(2^jL)^{-a}(2^{j+1}L)^{d/2+k}$, and sum the convergent geometric series.  Finitely many low modes are harmless because $\lambda_1>0$.

For the scalar kernel, apply this estimate with $a=r+1$ and split derivatives between $x$ and $y$ by Cauchy--Schwarz.  In particular, derivatives of order at most two in each variable have uniformly convergent series as soon as $r>d/2+1$.  The feature map and its differential converge uniformly in $\ell^2$ by the same tail bound with $k=0,1$.  Hence the limiting feature map is continuously differentiable and has bounded differential.

The matrix kernel has coefficient $q_r(\lambda_n)/\lambda_n\asymp\lambda_n^{-r-2}$ and one gradient already in each factor.  Use \eqref{eq:compact-derivative-tail} with $a=r+2$.  The smoothing threshold \eqref{eq:md-rstar} allows, in particular, two further derivatives in each variable, so the weighted source divergence can be taken termwise.  These estimates dominate all cutoffs because $0\le\chi\le1$.  They justify the cutoff-independent $C^1$ interaction and feature bounds used in \Cref{app:finite-mode-mean-field}.

\subsection{The heat approximation estimate}

Let $P_s=e^{-s\A_\pi}$ and let $Y_s^x$ be the diffusion with generator $-\A_\pi$, started at $x$.  We estimate its displacement in smooth Euclidean coordinates, which allow a direct application of It\^o's formula.  Choose a smooth embedding $J:\X\to\mathbb R^m$.  On a compact manifold its Euclidean chord distance and the Riemannian distance satisfy
\[
C^{-1}d_\X(x,y)\le\abs{J(x)-J(y)}\le C d_\X(x,y).
\]
The local comparison follows from injectivity of the differential; compactness and injectivity give the lower bound for pairs away from the diagonal.  It\^o's formula for the finitely many smooth coordinate functions of $J$ gives a bounded drift and bounded quadratic variation density.  Thus, for $0<s\le1$,
\[
\sup_x\E d_\X(Y_s^x,x)^2
\le C\sup_x\E\abs{J(Y_s^x)-J(x)}^2\le Cs.
\]
For a $1$-Lipschitz $\phi$, it follows that
\[
\abs{P_s\phi(x)-\phi(x)}
\le\E d_\X(Y_s^x,x)\le C\sqrt s.
\]
This proves \eqref{eq:heat-lip-approx}.  The same heat semigroup is Feller; smoothness for positive times and the derivative spectral bounds justify the uniformly convergent expansion used in \Cref{prop:compact-separation}.

\section{Mean-field estimates for the full kernel and its cutoffs}
\label{app:finite-mode-mean-field}

The finite-time estimate has two parts.  First couple the interacting particles to independent population characteristics with matching initial positions.  Then estimate the sampling error of those independent characteristics in the kernel feature space.  Bounded spatial derivatives control the coupling, and bounded time derivatives of the feature paths control the sampling error over a whole interval.

We treat the full kernel and its cutoffs together, proving \Cref{prop:compact-transport,prop:finite-mode-finite-time} with the same constants.  Let $a$ denote either a cutoff $\Lambda\ge\lambda_1$ or the symbol $\infty$ for the full kernel.  Set $\chi_\infty=1$, $b_a=-\nabla_x\G^a_{r,\pi}$, and
\[
\Psi_a(x)=\bigl(\sqrt{q_r(\lambda_n)\chi_a(\lambda_n)}\varphi_n(x)\bigr)_{n\ge1},
\qquad
\chi_\Lambda(\lambda)=\chi(\lambda/\Lambda).
\]
All constants below are independent of $a$, $N$, and $T$.

\subsection{Uniform interaction and feature bounds}

By \Cref{prop:compact-regularity,app:compact-analysis},
\begin{align}
\label{eq:appendix-uniform-b}
\sup_a\left(\norm{b_a}_{L^\infty}
+\norm{\nabla_xb_a}_{L^\infty}+\norm{\nabla_yb_a}_{L^\infty}\right)&\le C,\\
\label{eq:appendix-uniform-feature}
\sup_{a,x}\left(\norm{\Psi_a(x)}_{\ell^2}
+\norm{D\Psi_a(x)}_{\mathrm{op}}\right)&\le C.
\end{align}
Uniform convergence of the derivative tails proves continuous differentiability of the full, infinite-dimensional feature map.  Its feature representation is
\begin{equation}
\label{eq:appendix-feature-representation}
\D^a_{r,\pi}(\mu,\nu)=\left\|\int\Psi_a\,\dd(\mu-\nu)\right\|_{\ell^2}.
\end{equation}

\subsection{Coupling on a compact manifold}

Let $X_i=X_i^{N,a}$ be the interacting particles and let the independent population characteristics $\bar X_i=\bar X_i^a$ solve
\begin{equation}
\label{eq:nonlinear-characteristics}
\dot{\bar X}_i(t)=\int b_a(\bar X_i(t),y)\rho_t^a(\dd y),
\qquad\bar X_i(0)=X_i(0).
\end{equation}
The population path is deterministic; therefore these characteristics are independent, with common law $\rho_t^a$.

We measure the coupling error after a smooth embedding $J:\X\to\mathbb R^m$, so squared distances can be differentiated even across the manifold's cut locus.  Compactness gives $d_\X(x,y)\asymp\abs{J(x)-J(y)}$, as explained in \Cref{app:compact-analysis}.  Put
\[
z_i=J(X_i),\quad\bar z_i=J(\bar X_i),\quad
B_a(J(x),J(y))=dJ_x\,b_a(x,y).
\]
The fields $B_a$ are uniformly bounded and Lipschitz in Euclidean chord distance on $J(\X)\times J(\X)$, by \eqref{eq:appendix-uniform-b} and smoothness of $J$.  In particular,
\[
\abs{B_a(z,w)-B_a(z',w')}\le L(\abs{z-z'}+\abs{w-w'}),
\qquad \abs{B_a}\le B.
\]

Define
\begin{equation}
\label{eq:coupling-error}
R_N(t)=\frac1N\sum_i\abs{z_i(t)-\bar z_i(t)}^2,
\end{equation}
and write the force fluctuation as
\begin{equation}
\label{eq:empirical-force-fluctuation}
\zeta_i(t)=\frac1N\sum_j B_a(\bar z_i(t),\bar z_j(t))
-\int B_a(\bar z_i(t),J(y))\rho_t^a(\dd y).
\end{equation}
The off-diagonal terms are sampling fluctuations; the self-interaction contributes one additional term of size $N^{-1}$.  Conditional on $\bar z_i(t)=z$, the terms with $j\ne i$ are independent and centered after subtracting
$\beta_t(z)=\int B_a(z,J(y))\rho_t^a(\dd y)$.  Precisely,
\[
\zeta_i=\frac1N\sum_{j\ne i}\bigl[B_a(z,\bar z_j)-\beta_t(z)\bigr]
+\frac{B_a(z,z)-\beta_t(z)}{N}.
\]
The first sum has conditional mean zero and second moment at most $4B^2(N-1)/N^2$; the last, deterministic conditional term has squared norm at most $4B^2/N^2$.  Their cross expectation vanishes.  This also covers $N=1$, and proves
\begin{equation}
\label{eq:force-fluctuation-variance}
\E\abs{\zeta_i(t)}^2\le\frac{4B^2}{N}.
\end{equation}

Let $e_i=\abs{z_i-\bar z_i}$ and $\bar e=N^{-1}\sum_i e_i$.  The velocity difference is bounded by $L(e_i+\bar e)+\abs{\zeta_i}$.  Differentiating \eqref{eq:coupling-error}, using $\bar e^2\le R_N$ and $2e_i\abs{\zeta_i}\le e_i^2+\abs{\zeta_i}^2$, gives the pathwise inequality
\begin{equation}
\label{eq:R-differential}
R_N'(t)\le (4L+1)R_N(t)+\frac1N\sum_i\abs{\zeta_i(t)}^2,
\qquad R_N(0)=0.
\end{equation}
Consequently, for every $T\ge0$,
\begin{align}
\label{eq:coupling-path-bound}
\E\sup_{0\le t\le T}R_N(t)
&\le e^{(4L+1)T}\int_0^T\frac1N\sum_i\E\abs{\zeta_i(s)}^2\,\dd s\\
&\le\frac{4B^2T}{N}e^{(4L+1)T}\le\frac{C}{N}e^{CT}.
\end{align}
With $\bar\mu_t^{N,a}=N^{-1}\sum_i\delta_{\bar X_i(t)}$, the feature Lipschitz bound and the equivalence of distances give
\begin{equation}
\label{eq:D-coupled-empirical}
\D^a_{r,\pi}(\mu_t^{N,a},\bar\mu_t^{N,a})
\le\frac1N\sum_i\norm{\Psi_a(X_i)-\Psi_a(\bar X_i)}_{\ell^2}
\le C\sqrt{R_N(t)}.
\end{equation}

\subsection{The time supremum of the empirical feature process}

The remaining error compares the independent empirical measure with its common law.  Set
\begin{equation}
\label{eq:feature-empirical-process}
Z_N(t)=\frac1N\sum_i\bigl[\Psi_a(\bar X_i(t))-\E\Psi_a(\bar X_i(t))\bigr].
\end{equation}
Then $\norm{Z_N(t)}_{\ell^2}=\D^a_{r,\pi}(\bar\mu_t^{N,a},\rho_t^a)$.  Independence and the bounded feature map imply
\begin{equation}
\label{eq:Z-fixed-time}
\E\norm{Z_N(0)}_{\ell^2}^2\le C/N.
\end{equation}
The feature paths are continuously differentiable in $\ell^2$ with derivative
\begin{equation}
\label{eq:feature-path-derivative}
F_i(t)=D\Psi_a(\bar X_i(t))\,v^a_{\rho_t^a}(\bar X_i(t)).
\end{equation}
Both factors are uniformly bounded.  Differentiation under the expectation is justified by this deterministic bound, and at each fixed time the $F_i(t)$ are independent identically distributed Hilbert-valued variables.  Therefore
\begin{equation}
\label{eq:Z-derivative-fixed-time}
\E\norm{\dot Z_N(t)}_{\ell^2}^2
=\frac1N\E\norm{F_1(t)-\E F_1(t)}_{\ell^2}^2\le C/N.
\end{equation}
The fixed-time variance alone does not control the time supremum.  Apply the fundamental theorem of calculus to the feature paths, then Cauchy--Schwarz:
\begin{equation}
\label{eq:H1-time-embedding}
\sup_{0\le t\le T}\norm{Z_N(t)}_{\ell^2}^2
\le2\norm{Z_N(0)}_{\ell^2}^2
+2T\int_0^T\norm{\dot Z_N(s)}_{\ell^2}^2\,\dd s.
\end{equation}
Hence
\begin{equation}
\label{eq:nonlinear-empirical-bound}
\E\sup_{0\le t\le T}\D^a_{r,\pi}(\bar\mu_t^{N,a},\rho_t^a)^2
\le\frac{C(1+T^2)}{N}.
\end{equation}
The squared triangle inequality now joins the coupling error \eqref{eq:D-coupled-empirical} to this sampling error.  Together with \eqref{eq:coupling-path-bound}, it proves the $Ce^{CT}/N$ bounds in \Cref{prop:compact-transport,prop:finite-mode-finite-time}.  Independence was used only across the population characteristics at a fixed time; their time dependence is controlled by \eqref{eq:H1-time-embedding}.

\subsection{The spectral resolution estimate}
\label{app:finite-resolution-proof}

To prove \Cref{prop:finite-resolution}, we need a smoothing operator that uses only retained modes and approximates every Lipschitz function uniformly.  Let $\theta$ be as in that proposition and put $\delta=\Lambda^{-1/2}$.  A closed manifold with smooth positive density is doubling and satisfies the short-time Gaussian and H\"older heat-kernel bounds for its smooth uniformly elliptic generator.  Theorem~3.4, equations~(3.10) and~(3.12), of \citet{coulhon2012heat}, applied to $L=\A_\pi$ and $f(u)=\theta(u^2)$, gives, for $\Lambda\ge1$, the kernel bounds
\begin{equation}
\label{eq:localized-cutoff-kernel}
\abs{K_\Lambda(x,y)}\le C_M\delta^{-d}
(1+d_\X(x,y)/\delta)^{-M},
\qquad\int K_\Lambda(x,y)\pi(\dd y)=1.
\end{equation}
Here $M>d+1$ is fixed; $f$ is smooth and constant near zero, as required by that theorem.  Integration in distance shells gives
\[
\sup_x\int d_\X(x,y)\abs{K_\Lambda(x,y)}\pi(\dd y)\le C\delta.
\]
Subtract $\phi(x)$ inside the kernel integral.  For a Lipschitz $\phi$ this yields
\begin{equation}
\label{eq:appendix-spectral-lip}
\norm{(\Id-\theta(\A_\pi/\Lambda))\phi}_{L^\infty}
\le C\Lambda^{-1/2}\Lip(\phi).
\end{equation}
If $\lambda_1\le\Lambda<1$, finitely many modes occur and the same bound follows by enlarging a target-dependent constant, after normalizing $\pi(\phi)=0$.  Constants are preserved exactly, so this normalization entails no loss.

The approximation error is now controlled at scale $\Lambda^{-1/2}$.  The retained part is controlled by the truncated discrepancy: for a centered signed measure $\sigma$,
\begin{align*}
\abs{\sigma(\theta(\A_\pi/\Lambda)\phi)}^2
&\le \D_{r,\pi}^\Lambda(\sigma)^2
\sum_{n:\,\theta(\lambda_n/\Lambda)\ne0}
\frac{\theta(\lambda_n/\Lambda)^2}{q_r(\lambda_n)}
\abs{\ip{\phi}{\varphi_n}_{L^2(\pi)}}^2\\
&\le C\Lambda^{r+1}\D_{r,\pi}^\Lambda(\sigma)^2\norm{\phi}_{L^2(\pi)}^2.
\end{align*}
We used $\chi=1$ on the support of $\theta$ and $\Lambda\ge\lambda_1>0$.  Together with \eqref{eq:appendix-spectral-lip}, this proves \Cref{prop:finite-resolution}.

\section{Euclidean kernels, cutoff duality, and Hilbert-space chain rules}
\label{app:euclidean-calculus}

The Euclidean argument needs both local control, to evaluate the force at particles, and control of mass at infinity. We obtain them in three stages. Ground-state conjugation gives eigenvalue bounds and an operator core, from which local kernel regularity follows. Bounded approximations to $|x|^2$ then turn the dual discrepancy into a moment bound. Finally, Hilbert-space chain rules identify the relative-energy calculation to which the transport assumption applies.

\subsection{Confinement, compact resolvent, and an operator core}

We first put the generator into a form where quadratic confinement can be compared directly with the harmonic oscillator. The unitary map
\begin{equation}
\label{eq:ground-state-unitary}
(\mathcal U f)(x)=Z^{-1/2}e^{-V(x)/2}f(x)
\end{equation}
conjugates the generator to
\begin{equation}
\label{eq:Schrodinger-transform}
\mathcal U\A_\pi\mathcal U^{-1}
=\mathcal L=-\Delta+V_{\rm eff},\qquad
V_{\rm eff}=\tfrac14|\nabla V|^2-\tfrac12\Delta V.
\end{equation}
The assumptions give
\[
c|x|^2-C\le V_{\rm eff}(x)\le C(1+|x|^2).
\]
Every derivative of $V_{\rm eff}$ has polynomial growth; derivatives of positive order grow at most linearly. The quadratic lower bound and the Rellich compactness theorem, applied on balls and combined with the bound on $\int |x|^2|g|^2\,\dd x$, give compact embedding of the form domain into $L^2(\dd x)$. Hence $\mathcal L$ has compact resolvent. The form inequality
$\mathcal L\ge-\Delta+c|x|^2-C$ and the min--max principle compare its eigenvalues to those of the harmonic oscillator. Those oscillator eigenvalues are $\sqrt c(2|\alpha|+d)$, $\alpha\in\mathbb N_0^d$. Counting multi-indices proves \eqref{eq:euclidean-counting}. In particular,
\[
\operatorname{Tr}\Q_{r,\pi}
=\sum_{n\ge1}\frac1{\lambda_n(1+\lambda_n)^r}<\infty
\quad(r>d-1).
\]

To identify the dual discrepancy using compactly supported tests, we next approximate eigenfunctions in the required graph norms. The needed tail control follows from a weighted bootstrap: each eigenfunction $g$ of $\mathcal L$ has all polynomially weighted derivatives in $L^2(\dd x)$. For a smooth bounded truncated weight $w$, multiplication of $(\mathcal L-\lambda)g=0$ by $w^2g$ gives
\begin{equation}
\label{eq:euclidean-weighted-eigenfunction}
\int\left(|\nabla(wg)|^2+(V_{\rm eff}-\lambda)w^2|g|^2\right)\dd x
=\int|\nabla w|^2|g|^2\,\dd x.
\end{equation}
Use truncated versions of $\langle x\rangle^m$ whose derivatives are bounded by $C_m\langle x\rangle^{m-1}$. The quadratic lower bound and induction on $m$ give $\langle x\rangle^m g\in L^2$ for every integer $m$, and then the same identity gives weighted first derivatives. Higher derivatives follow by differentiating the equation: for a multi-index $\alpha$,
\[
-\Delta D^\alpha g
=(\lambda-V_{\rm eff})D^\alpha g
-\sum_{0<\beta\le\alpha}\binom{\alpha}{\beta}
(D^\beta V_{\rm eff})D^{\alpha-\beta}g.
\]
On unit balls the interior estimate for $-\Delta$ controls two more derivatives by the displayed right side and the function. The coefficient bounds on each ball are polynomial in the center. Multiplying the estimates by arbitrary polynomial weights, summing over a bounded-overlap unit-ball cover, and using the already established weighted lower derivatives completes the induction. All integrations can first be made with compact cutoffs and then passed to the limit using the preceding induction step.

These weighted estimates make spatial truncation converge in every integer graph norm. Let $\chi_R(x)=\chi(x/R)$ with $\chi\in C_c^\infty$, $\chi=1$ near the unit ball. Expanding $\mathcal L^k(\chi_Rg)-\chi_R\mathcal L^kg$ by Leibniz gives finitely many polynomially bounded coefficients times weighted derivatives of $g$, supported outside the ball of radius $R$. The preceding weighted bounds show
\[
\chi_Rg\longrightarrow g
\quad\hbox{in }\operatorname{Dom}(\mathcal L^k)
\quad\hbox{for every integer }k.
\]
After conjugation, $\chi_R\varphi_n\to\varphi_n$ in every integer graph norm of $\A_\pi$. Subtracting the $\pi$ mean preserves this convergence. Finite eigenfunction sums are a core by the spectral theorem; approximation of each such sum by these cutoffs and interpolation between integer graph norms prove that centered $C_c^\infty$ tests are a core for every positive fractional power used here.

\subsection{Local evaluation and the scalar and matrix kernels}

The kernels can now be constructed as Riesz representatives of point evaluation. Smoothing makes these evaluation maps continuous, and compactness of their images will give uniform convergence of the differentiated series on compact sets.

For a compact set $K$ and integer $k$, interior elliptic estimates for $\A_\pi$ on a larger compact set give
\begin{equation}
\label{eq:euclidean-local-elliptic}
\|h\|_{H^{2k}(K)}\le C_{K,k}\|(\Id+\A_\pi)^kh\|_{L^2(\pi)}.
\end{equation}
These estimates follow by iterating the second-order interior estimate; the smooth positive density is bounded above and below on the larger compact set. Choose $k$ with
$2k>d/2+2$ and $2k\le r+1$, possible under \eqref{eq:euclidean-r-threshold}. Sobolev embedding, with a positive margin, makes evaluation of every derivative of order at most two a continuous functional on $\mathcal H$, locally H\"older-continuous in the evaluation point in the operator norm. Denote its Riesz representative by $F_{x,\alpha}\in\mathcal H$. In particular,
\[
F_{x,0}=\G_{r,\pi}(\cdot,x),\qquad
\G_{r,\pi}(x,y)=\langle F_{x,0},F_{y,0}\rangle_{\mathcal H}.
\]
The orthonormal basis of $\mathcal H$ is
$\{q_r(\lambda_n)^{1/2}\varphi_n\}_{n\ge1}$, so the Riesz coefficients give exactly the stated kernel series. The images of a compact set under $x\mapsto F_{x,\alpha}$ are compact subsets of $\mathcal H$. Orthogonal spectral projections converge uniformly on such compact subsets. Cauchy--Schwarz then proves locally uniform convergence of the differentiated scalar series, including mixed derivatives with up to two derivatives in each variable. The same argument for the multiplier $q_r(\lambda)/\lambda$, which has still greater smoothing, gives the derivatives required for the matrix lift. Applying $\divpi\nabla\varphi_n=-\lambda_n\varphi_n$ termwise proves the Stein--Green identity.

Tonelli's theorem gives
\[
\int\G_{r,\pi}(x,x)\,\dd\pi(x)=\sum_{n\ge1}q_r(\lambda_n),
\]
and therefore the initial diagonal estimate. The local nature of these estimates matters. If $\widetilde G$ is the kernel after conjugation, then
\begin{equation}
\label{eq:kernel-transform}
\G_{r,\pi}(x,y)=Z e^{(V(x)+V(y))/2}\widetilde G(x,y).
\end{equation}
No polynomial tail bound for the left side follows from a polynomial bound for $\widetilde G$.

\subsection{Riesz representation of empirical and population measures}

Point evaluation and the test-function core connect the dual norm to the particle and population formulas. For a finite empirical measure $\mu$, the functional $h\mapsto\mu(h)$ is continuous on $\mathcal H$ by local evaluation. For $\rho=f\pi$ with $f\in L^2(\pi)$, it is continuous by the gap and Cauchy--Schwarz. The dense test-function core therefore identifies their centered Riesz representatives with
\[
m_\mu=\frac1N\sum_i\G_{r,\pi}(\cdot,X_i),\qquad
m_\rho=\Q_{r,\pi}(f-1).
\]
The dual norm is $\|m_\mu-m_\rho\|_{\mathcal H}$ and has the spectral expression in the main text. For bounded densities the ordinary kernel double integrals are absolutely convergent: positivity gives
$|\G(x,y)|\le\sqrt{\G(x,x)\G(y,y)}$, and the diagonal trace makes $\int\sqrt{\G(x,x)}f(x)\,\dd\pi(x)$ finite. The same assertion for empirical measures is a finite sum. Also $\int\G(x,y)\,\dd\pi(y)=0$, either by the Riesz representation of the zero centered functional or by this integrability and spectral approximation.

The relative-energy calculation also requires $h=m_\mu-m_\rho$ to lie in $\mathcal H_+$, so that $B_t h$ is an admissible test under \Cref{ass:euclidean-transport}. This extra regularity follows from the smoothing multiplier. Indeed, local elliptic estimates also make evaluation continuous on the centered space $\operatorname{Dom}((\Id+\A_\pi)^{r/2})$. Thus for every fixed $x$,
\[
\sum_{n\ge1}(1+\lambda_n)^{-r}|\varphi_n(x)|^2<\infty.
\]
Since
$(1+\lambda)^{r+2}q_r(\lambda)^2\le C(1+\lambda)^{-r}$ on the positive spectrum, $\G(\cdot,x)\in\mathcal H_+$. Likewise $\Q(f-1)\in\mathcal H_+$ for $f\in L^2(\pi)$. Therefore $h=m_\mu-m_\rho\in\mathcal H_+$.

If $B_t:\mathcal H_+\to\mathcal H$ has the assumed mapping property, apply the Riesz identity to $B_th$. This identity agrees with the original measure action: convergence in $\mathcal H$ implies local uniform convergence for the atoms and $L^2(\pi)$ convergence for the bounded density. The function $v_t\cdot\nabla h$ is integrable against $\rho_t$, since $v_t,f_t$ are bounded and $\nabla h\in L^2(\pi)$. Centering changes no action of $\mu-\rho_t$. Thus the Hilbert-space pairing is exactly the measure pairing in \eqref{eq:euclidean-actual-pairing}.

\subsection{Monotone moment cutoffs in the energy space}

The moment argument needs bounded tests that increase to $|x|^2$ while converging in $\mathcal H$. Concavity produces the monotonicity; Gaussian tails will give graph-norm convergence. Choose a smooth nondecreasing concave $\psi:[0,\infty)\to[0,\infty)$ with $\psi(s)=s$ for $s\le1$ and $\psi$ constant for $s\ge2$. Such a function is obtained by integrating a smooth nonincreasing function equal to one near $[0,1]$ and zero on $[2,\infty)$. Set
\[
p(x)=|x|^2,\qquad p_R(x)=R^2\psi(|x|^2/R^2),\qquad R\ge1.
\]
Concavity and $\psi(0)=0$ imply $\psi(s)-s\psi'(s)\ge0$, so $p_R\uparrow p$ as $R\to\infty$. Each $p_R$ is constant outside a ball, and $p_R-p$ and all of its derivatives vanish on $|x|\le R$. For every fixed derivative order, their magnitudes are bounded by a fixed polynomial in $|x|$, independently of $R\ge1$.

The coefficients of $\A_\pi^k$ and their required derivatives have polynomial growth. Consequently $\A_\pi^kp$ has polynomial growth and belongs to $L^2(\pi)$; it need not itself be a polynomial. Applying $\A_\pi^k$ to $p_R-p$ gives a finite sum of polynomially bounded terms supported on $|x|\ge R$. The Gaussian tail of $\pi$ shows that each sum tends to zero in $L^2(\pi)$. The same argument with compact cutoffs first shows that $p$ and $p_R$ belong to the relevant operator domains, using closedness of the operator and the compact-test core proved above. Thus
\[
p_R-p\longrightarrow0
\quad\hbox{in }\operatorname{Dom}((\Id+\A_\pi)^k)
\quad\hbox{for every fixed integer }k.
\]
Choose $2k\ge r+1$, center the functions, and use the spectral theorem to obtain convergence in $\mathcal H$. This is exactly the bounded-test approximation needed for the moment proof; it precedes, and does not presume, integrability of $p$ against the unknown measure.

\subsection{Verification of the Euclidean chain-rule hypotheses}

We now check the hypotheses of the chain rules in \Cref{app:atomic-regularization}. The local differentiated evaluation bounds make $x\mapsto F_{x,0}$ continuously differentiable as an $\mathcal H$-valued map. Hence every local particle solution has a differentiable empirical representative and satisfies \eqref{eq:finite-rank-diff}. The target-energy identity \eqref{eq:common-energy} and the moment bound then give global continuation.

For the population, \Cref{ass:euclidean-transport} bounds $f_t$ and $v_t$ on each finite interval. Consequently the functional $\phi\mapsto\int f_t v_t\cdot\nabla\phi\,\dd\pi$ is bounded by $C_T\|\phi-\pi(\phi)\|_{\mathcal H}$. The test-function core proved above therefore permits the extension of the weak continuity equation used in \eqref{eq:population-mean-derivative}. The population kernel mean is locally absolutely continuous, and the common-energy chain rule applies. All population pairings are integrable by the boundedness of $f_t,v_t$ and the Dirichlet-form bound; empirical pairings are finite sums.

For the relative energy, the preceding domain argument gives $h_t=m_{\mu_t}-m_{\rho_t}\in\mathcal H_+$ and identifies the measure action with the Riesz pairing in \eqref{eq:euclidean-actual-pairing}. Thus \eqref{eq:relative-energy} applies to the empirical error. The estimate in \Cref{prop:euclidean-transport} controls its transport term and gives finite-time comparison.

It remains to justify the spatial cutoff in the entropy calculation of \Cref{app:entropy-chain-rule}. For fixed $\eps>0$, the assumed finite-time $C_b^1$ bounds on $f_t$ and $v_t$ bound the regularized entropy terms. Using the cutoffs $\chi_R$ above, the boundary contributions vanish by the Gaussian tails of $\pi$, while $\nabla f_t\cdot v_t$ is integrable in time and space. One may therefore remove the spatial cutoff and then send $\eps$ to zero by the dominated-convergence argument in \Cref{app:entropy-chain-rule}. This gives the Euclidean entropy identity and completes the dynamical verification used in \Cref{thm:euclidean-main}.

\section{From empirical measures to labelled marginals}
\label{app:labelled-marginals}

Exchangeability turns labelled particles into sampling without replacement from their empirical measure.  Independent draws from that measure use replacement, so the two tuples can be coupled until a sampled label repeats.  This gives the small correction needed to pass from empirical convergence to propagation of chaos.

Let $(X_1^N,\ldots,X_N^N)$ be exchangeable, with empirical measure
$\mu^N=N^{-1}\sum_i\delta_{X_i^N}$, and let $\nu$ be a deterministic probability measure with finite first moment.  Fix a base point $x_0\in\X$, and equip $\X^\ell$ with the additive distance
\begin{equation}
\label{eq:product-distance}
d_\ell(x,y)=\sum_{k=1}^\ell d_\X(x_k,y_k).
\end{equation}
Write $W_{1,\ell}$ for its Wasserstein distance.

\begin{proposition}[Empirical-to-labelled comparison]
\label{prop:labelled-comparison}
For $1\le\ell\le N$,
\begin{equation}
\label{eq:labelled-comparison}
W_{1,\ell}\bigl(\Law(X_1^N,\ldots,X_\ell^N),\nu^{\otimes\ell}\bigr)
\le \ell\,\E W_1(\mu^N,\nu)+\varepsilon_{N,\ell},
\end{equation}
where
\begin{equation}
\label{eq:unbounded-labelled-error}
\varepsilon_{N,\ell}\le
\frac{\ell(\ell-1)}{N}\,
\E\int d_\X(x,x_0)\mu^N(\dd x).
\end{equation}
On a compact state space one may instead take
\begin{equation}
\label{eq:compact-labelled-error}
\varepsilon_{N,\ell}\le
\frac{\ell(\ell-1)}{2N}\operatorname{diam}(\X).
\end{equation}
\end{proposition}

\begin{proof}
We first couple sampling with and without replacement, then compare the with-replacement law with $\nu^{\otimes\ell}$.  Independently of the configuration, draw $I_1,\ldots,I_\ell$ independently and uniformly from $\{1,\ldots,N\}$.  Construct distinct indices $J_k$ sequentially.  If $I_k\notin\{J_1,\ldots,J_{k-1}\}$, set $J_k=I_k$; otherwise draw $J_k$ uniformly from the remaining indices.  Conditional on the previous $J$'s, each unused index has probability
\[
\frac1N+\frac{k-1}{N}\frac1{N-k+1}=\frac1{N-k+1}
\]
of becoming $J_k$.  Thus $J$ is uniform sampling without replacement, and the probability of replacement at step $k$ is $(k-1)/N$.

Conditional on a deterministic configuration, the index construction is invariant under every permutation of the $N$ indices.  Consequently, on the replacement event, both $I_k$ and $J_k$ have uniform marginal distributions on the $N$ labels.  The triangle inequality through $x_0$ therefore bounds the conditional expected cost at step $k$ by
\[
\frac{2(k-1)}{N}\int d_\X(x,x_0)\mu^N(\dd x).
\]
Summing and averaging proves \eqref{eq:unbounded-labelled-error}.  If $\X$ is compact, use its diameter on each replacement event to obtain \eqref{eq:compact-labelled-error}.

The with-replacement tuple has conditional law $(\mu^N)^{\otimes\ell}$.  The without-replacement tuple has unconditional law $\Law(X_1^N,\ldots,X_\ell^N)$ by exchangeability.  Finally, tensorizing a coupling between $\mu^N$ and $\nu$ gives
\begin{equation}
\label{eq:product-W-bound}
W_{1,\ell}((\mu^N)^{\otimes\ell},\nu^{\otimes\ell})
\le\ell W_1(\mu^N,\nu).
\end{equation}
Average this inequality and combine it with the constructed coupling to obtain \eqref{eq:labelled-comparison}.
\end{proof}

Apply the proposition with $\nu=\pi$ at time $t_N$.  The empirical last-iterate conclusion and the uniform first-moment bound make the right-hand side tend to zero.  Applying it with $\nu=\rho_t$ and taking the supremum over $t$ gives the labelled-marginal form of uniform-in-time propagation of chaos.  The replacement step itself requires only a uniform expected first moment; the stronger moment hypothesis is used earlier to obtain empirical $W_1$ convergence on an unbounded state space.

\section*{Acknowledgment of AI assistance}
The authors used generative artificial intelligence tools to assist in developing mathematical arguments and proofs and in drafting and revising the manuscript. The authors take full responsibility for all content of this work, including the validity of the mathematical arguments and the accuracy and attribution of the text and references.

\bibliographystyle{abbrvnat}
\bibliography{green_bessel_svgd_draft}

\end{document}